\documentclass[a4paper,12pt,reqno]{amsart}
\usepackage{a4wide}

\usepackage{amsmath,amssymb,amsthm,mathtools}

\usepackage[english]{babel}
\usepackage[colorlinks,citecolor=green,linkcolor=red]{hyperref}
\usepackage{amsthm}
\usepackage{color}
\usepackage{mathrsfs}
\usepackage{amsmath}
\usepackage{mathtools}
\usepackage{amsfonts}
\usepackage{amssymb}
\usepackage{bm}
\usepackage{physics}
\usepackage{enumitem}
\usepackage{tikz-cd}
\usepackage{a4wide}
\usepackage{cleveref}
\usepackage{esint}
\usepackage{nicefrac}
\usepackage{yfonts}
\usepackage{stmaryrd}
\usepackage[margin=2cm]{geometry}
\usepackage{comment}
\usepackage{mathrsfs}
\usepackage{dutchcal}
\usepackage{xfrac}

\usepackage{listings}

\numberwithin{equation}{section}

\theoremstyle{plain}
\newtheorem{thm}{Theorem}[section]

\newtheorem{prop}[thm]{Proposition}

\newtheorem{lemma}[thm]{Lemma}
\newtheorem{theorem}[thm]{Theorem}

\theoremstyle{remark}
\newtheorem{remark}[thm]{Remark}

\newtheorem{ex}[thm]{Example}

\theoremstyle{definition}
\newtheorem{defn}[thm]{Definition}

\DeclareMathOperator{\Aut}{Aut}

\newcommand{\supp}{{\mathrm {supp\,}}}

\newcommand{\ZZ}{{\mathbb{Z}}}

\newcommand{\RR}{\mathbb{R}}

\newcommand{\NN}{\mathbb{N}}
\newcommand{\QQ}{\mathbb{Q}}

\newcommand{\Ext}{\operatorname{Ext}}

\def\XXint#1#2#3{{\setbox0=\hbox{$#1{#2#3}{\int}$}
		\vcenter{\hbox{$#2#3$}}\kern-.5\wd0}}

\newcommand{\defeq}{\vcentcolon=}

\newcommand{\GGN}{\sfrac{\Gamma}{\Gamma_N}}

\let\phi\varphi
\let\epsilon\varepsilon
\let\eps\varepsilon

\author{Camillo Brena}
\address{\parbox{\linewidth}{ETH Z\"urich.\\
	R\"amistrasse 101,
    8092 Z\"urich -- Switzerland\\[-4pt]\phantom{a}}}\email{camillo.brena@math.ethz.ch}
\author{Elia Bru\`{e}}
\address{\parbox{\linewidth}{Bocconi University, Department of Decision Sciences.\\
	Via Sarfatti 25,
	20136 Milano -- Italy\\[-4pt]\phantom{a}}}
\email{elia.brue@unibocconi.it}
\begin{document}
	\title[Ollivier--Ricci Curvature on Groups of Polynomial Growth]{Ollivier--Ricci Curvature on \\Groups of Polynomial Growth}

	\maketitle
\setcounter{tocdepth}{4}

\begin{abstract}
We study Ollivier--Ricci curvature on Cayley graphs of groups of polynomial growth.
Our main result shows that non-negative Ollivier--Ricci curvature forces the group to be virtually abelian. As an application, we prove that connected vertex-transitive graphs of polynomial growth and non-negative Ollivier--Ricci curvature are quasi-isometric to $\ZZ^k$, for some $k\in\NN$.
\end{abstract}

\tableofcontents

\section{Introduction}
Let $\Gamma$ be a finitely generated group and let $S\subseteq\Gamma$  be a fixed finite symmetric set of generators.  We consider the  Cayley graph of $\Gamma$ with respect to $S$, i.e., the undirected graph whose vertex set   is $\Gamma$ and such that $g\sim h$ if and only if $g^{-1}h\in S$.  In other words, the edges are $(g,gs)_{g\in\Gamma, s\in S}$.
Notice that to the Cayley graph is naturally associated a shortest-path distance $d$, which corresponds to the algebraic word-length function defined in \eqref{brgfdsvcx} below. Of course, $d$ depends on $S$, but we will not make this dependence explicit in this introduction to keep the notation as light as possible.
In this introduction, we will always assume $e\notin S$, which, in particular, implies $g\nsim g$. Without this convention, \eqref{aaavefdsc},  \eqref{asavefdscz} and \eqref{aaascd} below have to be corrected.  This assumption will be dropped when we set up the precise framework.

\subsection{Ollivier--Ricci curvature}\label{veadscsdc}

Towards the definition of Ollivier--Ricci curvature, we  define the $1/2$-lazy symmetric random walk originating from $g$ with respect to $S$:

\begin{equation}\label{aaavefdsc}
    X_g=
    \begin{cases}
        g\qquad&\text{with probability $\frac{1}2$}\,,\\
        k\qquad&\text{with probability $\frac{1}{2|S|}$, for $k\sim g$}\,.
    \end{cases}
\end{equation}
The following definition has been given by Ollivier, \cite{ollivier2007ricci,ollivier2009ricci,ollivier2010survey}.
\begin{defn}\label{ordefn} The Ollivier--Ricci curvature is defined as
    \begin{equation}
    \kappa(g,h)\defeq 1-\frac{\inf_{(X_g,X_h)}\mathbb{E} \big(d(X_g,X_h)\big)}{d(g,h)}\qquad\text{for $g\ne h$}\,.
\end{equation}
\end{defn}
In the definition above, the infimum is taken among all couplings of $X_g$ and $X_h$.   From the theory of optimal transport,  it is well known that the infimum above is attained, and the corresponding couplings are called optimal.

Notice that $\kappa$ depends heavily on $S$, as both the distance $d$ and the random variables $X$ depend on $S$. To keep notation in this introduction simple, we omit to write this dependence.

An equivalent perspective is the following. Denoting by $\mu_g$  the law of  $X_g$, namely \begin{equation}
    \mu_g(k)\defeq\label{asavefdscz}
    \begin{cases}
        \frac{1}2\qquad&\text{if $k=g$}\,,\\
        \frac{1}{2|S|}\qquad&\text{if $k\sim g$}\,,\\
        0\qquad&\text{otherwise}\,,
    \end{cases}
\end{equation}
 we have
\begin{equation}
    \kappa(g,h)= 1-\frac{W_1(\mu_g,\mu_h)}{d(g,h)}\qquad\text{for $g\ne h$}\,.
\end{equation}
As usual, $W_1$ is the $1$-Wasserstein distance with respect to $d$, which, for probability measures $\mu,\nu$, is defined as
\begin{equation}
    W_1(\mu,\nu)\defeq \inf_{\pi} \sum_{g,h\in\Gamma} d(g,h)\pi(g,h)\,,
\end{equation}
where the infimum is taken among all transport plans $\pi$ for $\mu,\nu$, i.e., those probability measures $\pi$ on $\Gamma\times\Gamma$  satisfying
\begin{equation}
    \sum_{h\in \Gamma}\pi(g,h)=\mu(g)\quad\text{and}\quad\sum_{h\in \Gamma}\pi(h,g)=\nu(g)\qquad\text{for every $g$}\,.
\end{equation}
\begin{defn}\label{ordef1}
    A  graph is said to have non-negative Ollivier--Ricci curvature  if 
\begin{equation}
    \kappa(g,h)\ge 0\qquad\text{for   $g\sim h$}\,.
\end{equation}
\end{defn}
Notice that
\begin{equation}
    \kappa(g,h)\ge 0\quad\text{if and only if}\quad W_1(\mu_g,\mu_h)\le 1\qquad\text{for  every $g\sim h$}\,,
\end{equation}
which follows immediately from the definition.

\begin{remark}
    Another notion of Ricci curvature for graphs has been defined by Lin--Lu--Yau in  \cite{lin2011ricci}. For $\alpha\in (0,1)$, they first defined $\kappa_\alpha$ by replacing, in the definition of $\kappa$,   the $1/2$-lazy symmetric random walk $X_g$ (or its law $\mu_g$)  by the $\alpha$-lazy symmetric random walk (or its law). Of course, $\kappa_{1/2}$ corresponds to Ollivier's notion.  Then, Lin--Lu--Yau's curvature is defined as  $\lim_{\alpha\uparrow 1}\kappa_\alpha/(1-\alpha)$ (the limit exists thanks to the concavity of $\alpha\mapsto\kappa_\alpha$, \cite{lin2011ricci}). It follows from \cite[Proposition 2]{loisel2014ricci} that having non-negative Ollivier--Ricci curvature and having non-negative Lin--Lu--Yau--Ricci curvature are in fact equivalent.
\end{remark}

Given the flexibility in the choice of the optimal plan, one may wonder whether every Cayley graph has non-negative Ollivier--Ricci curvature. Of course, this is not the case, as the following three examples show.
\begin{ex}\label{vefdscvc}
        Consider the (finite) dihedral group corresponding to the symmetries of the hexagon, $\Gamma\defeq \langle s,t\,|\, t^6=e, s^2=e, sts=t^{-1} \rangle$ and let $S\defeq\{s,st,st^3\}$, which is a finite symmetric set of generators. A tedious  computation shows that $\kappa(e,st)<0$. Hence, the Cayley graph of $\Gamma$ with respect to $S$  does not have non-negative Ollivier--Ricci curvature. 
    \end{ex}

\begin{ex}\label{ex:heisen}
    If $\Gamma\defeq\langle x_1,x_2\,|\,y\defeq [x_1,x_2]\text{ is central}\rangle=H_3(\ZZ)$ is the Heisenberg group, for every finite  symmetric set of generators $S$, the Cayley graph of $\Gamma$ with respect to $S$  does not have non-negative Ollivier--Ricci curvature.  This is an immediate consequence of our Theorem \ref{thm:main} below.
\end{ex}
\begin{ex}
    Let $\Gamma\defeq \langle a,b\rangle$, the free group on two generators, where $S\defeq\{a,a^{-1},b,b^{-1}\}$ is a finite symmetric set of generators. It is easy to realize that the Cayley graph of $\Gamma$ with respect to $S$  does not have non-negative Ollivier--Ricci curvature. Actually, from the main result of the forthcoming \cite{BHM}, for every finite  symmetric set of generators $S$, the Cayley graph of $\Gamma$ with respect to $S$  does not have non-negative Ollivier--Ricci curvature. 

\end{ex}
On the other hand, we also have plenty of Cayley graphs of non-negative Ricci curvature.
\begin{ex}\label{verdsc}Let $\Gamma$ be a finitely generated group and let $S$ be a finite symmetric set of generators. Assume one of the following. 
\begin{itemize} 
    \item  $\Gamma$ is abelian. 
    \item $S$  is  conjugation invariant, meaning that $tst^{-1}\in S$ for every $t,s\in S$.
    \item  $\Gamma$ is finite and $S=\Gamma\setminus\{e\}$, meaning that the Cayley graph of $\Gamma$ with respect to $S$ is a complete graph.
\end{itemize}
Then, it is easy to verify that the Cayley graph of $\Gamma$ with respect to $S$ has non-negative Ricci curvature.
\end{ex}
By this example, we  have an abundance of Cayley graphs with non-negative Ollivier--Ricci curvature. One may wonder whether the cases above are exhaustive.  The answer is negative, as the next example shows.
\begin{ex}\label{infdie}
Consider the infinite dihedral group  $\Gamma\defeq \langle s,t\,|\, s^2=e, sts=t^{-1}\rangle$, with $S\defeq \{s,t,t^{-1}\}$, which is a finite symmetric set of generators. Notice that any element of $\Gamma$ can be written either as $st^j$ or $t^j$, for $j\in\ZZ$.

A direct computation shows that the Cayley graph of $\Gamma$ with respect to $S$ has non-negative Ricci curvature.  Moreover, $\Gamma$ is infinite, virtually abelian ($\Gamma_N\defeq \langle t\rangle$ is cyclic and of finite index), but not abelian. Finally, $\Gamma$ does not possess any finite conjugation invariant set of generators. Indeed,  any set of generators must contain $st^j$ for some $j$, which satisfies $t^{k}st^j t^{-k}=st^{j-2k}$, and hence has infinite conjugacy class.
\end{ex}

Thus, no  classification of groups admitting Cayley graphs with non-negative Ollivier--Ricci curvature  can be extracted from Example \ref{verdsc}. In our first main result, stated below, we give the correct characterization of such groups for finitely generated virtually nilpotent groups.

\begin{theorem}\label{thm:main}
Let $\Gamma$ be a finitely generated virtually nilpotent group. Then, the following are equivalent
\begin{enumerate}
    \item $\Gamma$ is virtually abelian, 
    \item There exists a finite symmetric set of generators $S$ such that the Cayley graph of $\Gamma$ with respect to $S$  has non-negative Ollivier--Ricci curvature.
\end{enumerate}
\end{theorem}
We recall that $\Gamma$ is virtually nilpotent  (resp.\ virtually abelian) if there exists a  nilpotent (resp.\ abelian) normal subgroup $\Gamma_N\unlhd\Gamma$ of finite index, i.e.,   $\big|\GGN\big|<\infty$. The class of virtually nilpotent groups is relevant and well motivated in this setting.  Indeed, the celebrated Gromov's Theorem \cite{MR623534} states that a finitely generated group $\Gamma$ is virtually nilpotent if and only if it has polynomial growth. Polynomial growth means that some (hence all) Cayley graph of the group has polynomial growth as a metric space.  

Moreover, in the forthcoming \cite{BHM}, it is  proved that groups admitting a Cayley graph of non-negative Ollivier--Ricci curvature have polynomial growth. This not only makes the setting natural, but implies that the virtually nilpotent assumption in Theorem \ref{thm:main} can in fact be dropped, thus yielding a complete characterization of finitely generated groups admitting Cayley graphs of non-negative Ollivier--Ricci curvature.

\subsection{Ollivier--Ricci curvature at large scales}\label{largescales}
A somewhat unpleasant feature of the Ollivier--Ricci curvature is that it depends on the set of generators. In particular, a group may have, at the same time, Cayley graphs  with and without non-negative Ollivier--Ricci curvature. This happens for instance for the finite dihedral group  of Example \ref{vefdscvc} (which is of course virtually abelian), recall also Example \ref{verdsc}. 
On the other hand, we have seen in Theorem \ref{thm:main} that those finitely generated virtually nilpotent groups admitting a Cayley graph with non-negative Ollivier--Ricci curvature are precisely the virtually abelian groups. It is then natural to ask whether virtually abelian groups can be actually recognized by looking at some notion of curvature, in a way that does not depend on the chosen set of generators.
In other words, the question is whether there exists some quantity  (related to Ollivier--Ricci curvature) which is suitable to tell apart the purely virtually nilpotent case and the virtually abelian case, independently of the set of generators.  As we will argue below, it turns out that the correct notion is that of Ollivier--Ricci curvature at large scales, the intuition being that, asymptotically, the effect played by the choice of the generators should play a smaller role.

Another indication that this notion is more robust than plain Ollivier--Ricci curvature is the role that it has in the proof of Theorem \ref{thm:verttrans}, where we are going to need the characterization of finitely generated  groups of polynomial growth with non-negative Ollivier--Ricci curvature at large scales (but not necessarily with non-negative Ollivier--Ricci curvature). For details, see that discussion below Theorem \ref{thm:verttrans} and its proof.

\medskip

We  introduce  curvature at large scales. First, we denote by $X^n_g$  the $n$-step $1/2$-lazy symmetric random walk originating from $g$ with respect to $S$, i.e., 
\begin{equation}\label{aaascd}
    X^{n+1}_g=
    \begin{cases}
        X^n_g\qquad&\text{with probability $\frac{1}2$}\,,\\
        k\qquad&\text{with probability $\frac{1}{2|S|}$, for $k\sim X^n_g$}\,,
    \end{cases}
\end{equation}
with independent increments. It is easy to verify that the law of $X^n_g$ is $\mu^{\ast n}_g$, which is the left translation by $g$ of the $n$-times convolution of $\mu$, see Section \ref{sect:mk} for the precise definitions. 
\begin{defn}
    The Ollivier--Ricci curvature at large scales is defined as 
\begin{equation}
    \kappa_n(g,h)\defeq 1-\frac{\inf_{(X^n_g,X^n_h)}\mathbb{E} \big(d(X^n_g,X^n_h)\big)}{d(g,h)}\qquad\text{for $g\ne h$ and $n\in\NN$}\,.
\end{equation}
\end{defn}
In the equation above, the infimum is taken among all couplings of $X^n_g$ and $X^n_h$. As for $\kappa$, we remark that the infimum is attained, and the corresponding couplings are called optimal.
Equivalently,
\begin{equation}
    \kappa_n(g,h)\defeq 1-\frac{W_1(\mu^{\ast n}_g,\mu^{\ast n}_h)}{d(g,h)}\qquad\text{for $g\ne h$ and $n\in\NN$}\,.
\end{equation}

\begin{remark}\label{rem:vfdsc}
    Assume that $\Gamma$ has non-negative Ollivier--Ricci curvature. Then, $\Gamma$ has non-negative Ollivier--Ricci curvature  at large scales. More precisely,
    \begin{equation}\label{vedasdsc}
       \kappa_n(g,h)\ge 0\qquad\text{for every $g\ne h$ and $n\in\NN$}\,.
    \end{equation}

Indeed, \eqref{vedasdsc} is equivalent to 
\begin{equation}\label{vedasdsc1}
    \inf_{(X^n_g,X^n_h)}\mathbb{E} \big(d(X^n_g,X^n_h)\big)\le d(g,h)\qquad\text{for every $g\ne h$ and $n\in\NN$}\,.
\end{equation}
By the non-negativity of the Ollivier--Ricci curvature and the triangle inequality,  we see that \eqref{vedasdsc1} holds for $n=1$. 
Assume now that \eqref{vedasdsc1} holds for $1,\dots,n$. For $g,h\in \Gamma$, we  construct a coupling  $(X^{n+1}_g,X^{n+1}_h)$ as follows. We first consider an optimal coupling  $(X^{n}_g,X^{n}_h)$. Then, conditionally on $(X^{n}_g,X^{n}_h)=(g',h')$, we sample according to an optimal coupling for $(X_{g'},X_{h'})$, that is, 
\begin{equation}
    \mathbb{P}\big(X^{n+1}_g=g'',X^{n+1}_h=h'' |X^{n}_g=g',X^{n}_h=h'\big)=\mathbb{P}(X_{g'}=g'', X_{h'}=h'')\,.
\end{equation}
Hence, using the optimality of the couplings and the inductive assumption,
\begin{equation}
   \mathbb E \big(d(X^{n+1}_g,X^{n+1}_h) \big)=\mathbb E\big(\mathbb E\big(d(X^{n+1}_g,X^{n+1}_h)|(X_g^n,X_h^n)\big)\big)
   \le \mathbb E \big(d(X_g^n,X_h^n)\big) \le d(g,h)\,,
\end{equation}
so that \eqref{vedasdsc1} holds for $n+1$ as well. A similar argument can be carried out in the language of the laws $\mu^n$, where the pairing of the couplings reads as disintegration of optimal transport plans.
\end{remark}

We now state our main results in this direction, which give an exhaustive answer to the question introduced at the beginning of this subsection. Indeed, we are going to show that, for $\Gamma$  a finitely generated virtually nilpotent group with any finite symmetric set of generators $S$, the following happens. Either the Ollivier--Ricci curvature at large scales stays uniformly strictly negative along some directions, \eqref{vefdcs}, or it converges to $0$ in a very strong sense, \eqref{bbbcdscscd}. The first case happens in the purely virtually nilpotent case, the second in the virtually abelian regime. This is the outcome of the following two theorems.

\begin{theorem}\label{vecdscs}
Let $\Gamma$ be a finitely generated virtually nilpotent group and consider a finite symmetric set of generators $S$. Assume that $\Gamma$  is  not virtually abelian.  Then, there exists a sequence $(g_n)_n\subseteq\Gamma$ with $        \lim_{n\rightarrow\infty}\frac{d(e,g_n)}{n}\in (0,\infty)$ such that
    \begin{equation}\label{vefdcs}
        \limsup_{n\rightarrow\infty}\kappa_{n^2}(e,g_n)<0\,.
    \end{equation}
\end{theorem}
We remark that Theorem \ref{vecdscs} gives a complete answer to \cite[Problem C]{ollivier2010survey}. Moreover, we decided to state this theorem with elements $g_n$ with $d(e,g_n)\simeq n$ and time-steps $n^2$ to respect the natural parabolic scaling. 

\begin{theorem}\label{cdscs}
     Let $\Gamma$ be a finitely generated virtually abelian group and consider a finite symmetric set of generators $S$. Then
     \begin{equation}\label{bbbcdscscd}
         |\kappa_n(e,g)|\le \frac{C}{d(e,g)}\qquad\text{for every $g\ne e$ and $n\in\NN$}\,,
     \end{equation}
     where $C$ depends only on $\Gamma$ and $S$.
\end{theorem}

\subsection{Transitive graphs}
Let $G=(V,E)$ be an undirected graph. Here, $V$ is the set of vertices and $E$ is the set of edges. For $x,y\in V$, we  write $x\sim y$ if $x$ and $y$ are connected by an edge of $E$. We assume throughout that graphs are connected and have bounded degree.

A bijection  $\phi:V\rightarrow V$ is called a graph automorphism, and we write $\phi\in\Aut(G)$, if 
\begin{equation}
    \phi(x)\sim\phi(y)\text{ if and only if }x\sim y\qquad\text{for every $x,y\in V$}\,.
\end{equation}
Notice that  to $G=(V,E)$  is naturally associated a shortest-path distance $d$ and that graph automorphisms are precisely the isometries.
We say that a graph is vertex-transitive, or simply transitive, if the automorphism group acts transitively on the vertex set. Of course, Cayley graphs are transitive graphs (multiplications by elements of the groups are automorphisms), but the converse need not hold. However,   see the  discussion below Theorem \ref{thm:verttrans} for a partial converse.

\medskip

We have a natural generalization of the notion of Ollivier--Ricci curvature to this setting. We assume in this introduction that the graph has no self-loops, i.e.\ $x\nsim x$. This assumption serves only to simplify the notation and will later be removed. We  consider, as in \eqref{aaavefdsc}, the $1/2$-lazy symmetric random walk originating from $x$, namely
\begin{equation}
    X_x=
    \begin{cases}
        x\qquad&\text{with probability $\frac{1}2$}\,,\\
        y\qquad&\text{with probability $\frac{1}{2\deg(x)}$, for $y\sim x$}\,,
    \end{cases}
\end{equation}
where $\deg(x)\defeq|\{y\in G:y\sim x\}|$ is the degree of $G$ at $x$.
Ollivier--Ricci curvature can now be defined exactly as in Definitions \ref{ordefn} and \ref{ordef1}.

\medskip
\begin{theorem}\label{thm:verttrans}
	Let $G=(V,E)$ be a  transitive graph of polynomial growth. If $G$ has non-negative Ollivier--Ricci curvature, 	then it is quasi-isometric to $\ZZ^k$ for some $k\in\NN$.
\end{theorem}
 Recall that two metric spaces $(X,d_X)$ and $(Y,d_Y)$ are said to be {$(A,B)$-quasi-isometric} if there exists a function
  $f:(X,d_X)\rightarrow (Y,d_Y)$
  such that
 \begin{equation}\notag
   \frac{1}{A}d_X(x,x')- B\leq d_Y(f(x),f(x'))\leq  {A}d_X(x,x')+ B\qquad\text{for every $x,x'\in X$}\,,
 \end{equation}
 and
 \begin{equation}
   \sup_{y\in Y}d_Y(y,f(X))\leq B\,.
 \end{equation}
 Such a function is referred to as an  $(A,B)$-quasi-isometry.
 We say that two metric spaces are quasi-isometric if they are $(A,B)$-quasi-isometric for some $A,B<\infty$.

 \medskip

It is shown in the forthcoming \cite{BHM} that a  transitive graph of non-negative Ollivier--Ricci curvature has polynomial growth. This motivates Theorem \ref{thm:verttrans}, and, in particular, shows that the assumption of polynomial growth can indeed be dropped. 
Thanks to the results of Trofimov \cite{MR735714, MR811571} and Sabidussi \cite{Sabidussi},  we will prove that a connected transitive graph of polynomial growth has a quotient which can be embedded into a Cayley graph of a group of polynomial growth with non-negative Ollivier--Ricci curvature at large scales, as defined in Subsection \ref{largescales} (but it does  not, in general, have non-negative Ollivier--Ricci curvature).  Hence, by our results, such group is virtually abelian, and therefore the graph is quasi-isometric to $\ZZ^k$ for some $k\in\NN$. 
\subsection{Strategy of the proof}
\label{sec:strategy}

We now discuss the broad scheme of the proof of Theorem \ref{vecdscs}, which is the heart of Theorem \ref{thm:main}. It is instructive to start from the simple case $\Gamma=H_3(\ZZ)\defeq \langle x_1,x_2\,|\,y\defeq [x_1,x_2]\text{ is central}\rangle$, where $S\defeq \{x_1,x_1^{-1},x_2,x_2^{-1}\}$. We will argue that there exists $\delta\in (0,1)$ such that 
\begin{equation}\label{tobound}
    W_1(\mu^{\ast n^2}_e,\mu^{\ast n^2}_{x_2^n})\ge n(1+\delta)\qquad\text{for $n$ large enough}\,.
\end{equation}
There are two main ways to accomplish this:
\begin{itemize}
    \item Prove the bound directly at the level of the Cayley graph.
    \item Consider the blow-down of the metric space given by the Cayley graph, which is $H_3(\ZZ)$, and argue at the continuous limit.
\end{itemize}
We are going to follow the first alternative. We anyhow discuss the second  in Section \ref{sect:noemandb}. 

We  need a non-trivial lower bound on $ W_1(\mu^{\ast n^2}_e,\mu^{\ast n^2}_{x_2^n})$. By the easy part of Kantorovich duality, see Subsection \ref{set:prooffive}, for every  $1$-Lipschitz function $\Psi$,
\begin{equation}\label{vfedcs}
    W_1(\mu^{\ast n^2}_e,\mu^{\ast n^2}_{x_2^n})\ge\sum_{g\in\Gamma} \Psi(g)\mu^{\ast n^2}_{x_2^n}(g)- \ \sum_{g\in\Gamma} \Psi(g)\mu^{\ast n^2}_{e}(g)\,.
\end{equation}
We thus look for a suitable $1$-Lipschitz function $\Psi$ giving the bound
\begin{equation}
    \sum_{g\in\Gamma} \Psi(g)\mu^{\ast n^2}_{x_2^n}(g)- \ \sum_{g\in\Gamma} \Psi(g)\mu^{\ast n^2}_{e}(g)\ge n(1+\delta)\,.
\end{equation}

We fix the Malcev basis $x_1,x_2,y$,  as in Subsection \ref{subsec:Malcev}. Thus, every element $g\in\Gamma$ can be written uniquely in the form $g=x_1^{a}x_2^{b}y^{c}$, where $a,b,c\in\ZZ$. A natural choice of the function $\Psi$ is $\ell(g)\defeq b(g)$. This is $1$-Lipschitz and moreover
\begin{equation}
        \sum_{g\in\Gamma} \ell(g)\mu^{\ast n^2}_{x_2^n}(g)- \ \sum_{g\in\Gamma} \ell(g)\mu^{\ast n^2}_{e}(g)=\sum_{g\in\Gamma} b(g) \mu^{\ast n^2}_{x_2^n}(g)- \ \sum_{g\in\Gamma} b(g)\mu^{\ast n^2}_{e}(g)= n\,.
\end{equation}
This proves \eqref{tobound} with $\delta=0$, which is not  enough: we need $\delta>0$. The most natural choice is to exploit the non-abelianity of the group. The easiest terms that recognize this effect are $a(g)c(g)$ and $b(g)c(g)$. We then try with
\begin{equation}
    \Psi_\gamma(g)\defeq b(g)- \gamma a(g)c(g)\,,
\end{equation}
where $\gamma\in(0,1)$ has to be chosen carefully. It turns out that $\Psi_\gamma$ is not $1$-Lipschitz on $\Gamma$, but we will take care of this issue later.
We can anyhow compute
\begin{equation}\label{bgrfdv}
\begin{split}
     \sum_{g\in\Gamma} \Psi_\gamma(g)&\mu^{\ast n^2}_{x_2^n}(g)- \ \sum_{g\in\Gamma} \Psi_\gamma(g)\mu^{\ast n^2}_{e}(g)
     \\&= \sum_{g\in\Gamma} (b(g)- \gamma a(g)c(g))\mu^{\ast n^2}_{x_2^n}(g)-  \sum_{g\in\Gamma} (b(g) - \gamma a(g)c(g))\mu^{\ast n^2}_{e}(g)
     \\
     &=n-\gamma\sum_{g\in\Gamma} (a{(x_2^ng)}c{(x_2^ng)}-a(g)c(g))\mu^{\ast n^2}_e(g)
     \\
     &=n-\gamma \sum_{g\in\Gamma} a(g)(-na(g)+c(g)-c(g))\mu^{\ast n^2}_e(g)\\
     &=n\bigg(1+\gamma \sum_{g\in\Gamma}a(g)^2\mu^{\ast n^2}_e(g)\bigg)\,.
\end{split}
\end{equation}
Now, the task is to choose $\gamma$ and cut-off $\Psi_\gamma$ (everything depending on $n$) such that the modified potential is $1$-Lipschitz and such that it holds $\gamma_n \sum_{g\in\Gamma}a(g)^2\mu^{\ast n^2}_e(g)>\delta$, for some $\delta\in(0,1)$ independent of $n$.   This will be part of our discussion  in Subsections \ref{subsec:Kantorovich} and \ref{subsec:mainprop}. Once  this is accomplished, \eqref{tobound} follows from \eqref{vfedcs} and \eqref{bgrfdv}.

In the general case, we face three main difficulties:
\begin{enumerate}
    \item $\Gamma$ is not nilpotent, but only virtually nilpotent. That is, there exists  only $\Gamma_N\unlhd\Gamma$ nilpotent of finite index.
    \item We do not have  a simple set of generators as we had for $H_3(\ZZ)$.
    \item $\Gamma_N$ needs not to be as simple as $H_3(\ZZ)$, as a group.
\end{enumerate}
Item $(3)$ is the easiest to deal with: up to taking a suitable quotient, we can assume that  $\Gamma_N$ has step $2$. This is because the Ollivier--Ricci curvature does not decrease after taking quotients, see Subsection \ref{sec:reduction_step_2} for details.  This reduction is not really necessary, but simplifies the rest of the proof.
Items $(1)$ and $(2)$ are better dealt together at the same time. There are various ways to do that, ours is the following: the restriction of $d$ to $\Gamma_N$ is a left invariant distance on $\Gamma_N$. While, at definite scales, it does not have the nice properties that we used before, asymptotically, it resembles a much simpler distance resembling a word-length metric. This is discussed in Subsection \ref{subsec:auxiliary_distance}.

\subsection{Nilpotent structures and curvature}

The relation between nilpotent structures and spaces with uniform curvature bounds has a long history.
In 1978, Gromov proved his celebrated almost flat manifold theorem, later generalized and refined in \cite{BuserKarcher,Ruh}.
It states that a compact Riemannian manifold $(M^n,g)$ with sufficiently small sectional curvature $|\operatorname{Sec}_g|\le\varepsilon(n)$, and bounded diameter $\operatorname{diam}_g(M)\le 1$, is finitely covered by a nilmanifold, that is, by a quotient of a nilpotent Lie group by a cocompact lattice.
A basic example is the Heisenberg nilmanifold ${\rm Nil}^3\defeq H_3(\mathbb{R})/H_3(\ZZ)$, see \eqref{eq:heisenberg_R}.
Later, Fukaya developed a local version of this picture \cite{Fukaya88}, describing manifolds with bounded sectional curvature and sufficiently collapsed geometry in terms of fibration structures whose fibers carry nilpotent features.

It was then understood that the two-sided sectional curvature bound can be weakened to a lower curvature bound, while still retaining nilpotent structure at the level of the local fundamental group.
In \cite{FukayaYamaguchi}, Fukaya--Yamaguchi proved a generalized Margulis-type result showing that, under a lower sectional curvature bound $\operatorname{Sec}_g\ge -1$, the image of the local fundamental group $\pi_1(B_{\varepsilon(n)}(p))\to \pi_1(B_1(p))$
is virtually nilpotent.
Later, Kapovitch--Petrunin--Tuschmann \cite{KapovitchPetruninTuschmann} refined this analysis by proving quantitative bounds on the index and the nilpotency length.
Finally, Kapovitch--Wilking \cite{KapovitchWilking} established analogous results under lower Ricci curvature bounds.

In the setting of complete manifolds with non-negative Ricci curvature, the first results for fundamental groups go back to Milnor's 1968 work \cite{Milnor1968}, where he proved polynomial growth for finitely generated subgroups of the fundamental group.
Combined with Gromov's theorem on groups of polynomial growth \cite{MR623534}, this yields virtual nilpotency of every finitely generated subgroup of the fundamental group.
In the same 1968 paper, Milnor formulated his famous conjecture that the fundamental group of a complete manifold with non-negative Ricci curvature should always be finitely generated.
This conjecture was disproved  recently \cite{BrueNaberSemola,BrueNaberSemola6D}.

\medskip

While lower curvature bounds, such as $\operatorname{Sec}_g\ge -1$ or $\operatorname{Ric}_g\ge -(n-1)$, make nilpotent structures appear, non-negative curvature often forces these structures to further rigidify to abelian ones.
This class of phenomena is closer in spirit to the main result of the present paper, see Theorem \ref{vecdscs}.
One of the first instances of this phenomenon appears in Milnor's 1968 paper.
He observed that the fundamental group of the Heisenberg nilmanifold ${\rm Nil}^3\defeq H_3(\mathbb{R})/H_3(\ZZ)$ has quartic polynomial growth, due to its nilpotent structure.
This growth is faster than the general bound he proved for finitely generated subgroups of fundamental groups of complete $n$-dimensional manifolds with non-negative Ricci curvature, which is polynomial of degree at most $n$.
In particular, the Heisenberg nilmanifold cannot admit any Riemannian metric with non-negative Ricci curvature, even though it admits metrics with   sectional curvature arbitrarily  close to $0$.

More generally, Cheeger and Gromoll used their splitting theorem \cite{Cheeger-Gromoll-splitting} to show that the fundamental group of a compact manifold with non-negative Ricci curvature is virtually abelian.
The compactness assumption is fundamental, as shown by Wei \cite{Wei}; see also \cite{Wilking}.
Fukaya and Yamaguchi conjectured \cite{FukayaYamaguchi} that, under non-negative sectional curvature, the index of the abelian subgroup should be bounded by a constant depending only on the dimension.
This conjecture remains open.
Recently, Bruè--Naber--Semola disproved the analogous statement under non-negative Ricci curvature \cite{BrueNaberSemolaRicci}.

Finally, we mention Carnot--Carathéodory structures, which naturally appear as blow-downs of nilpotent structures; see Section \ref{sect:noemandb}.
They are known not to satisfy lower Ricci curvature bounds in the Lott--Sturm--Villani synthetic sense, see for instance \cite{Juillet2021,MagnaboscoRossi,RizziStefani2023,DriverMelcher2005,AmbrosioStefani2020,HuangSun2020,MagnaboscoRossi1}, even in basic examples such as the Heisenberg group.
This is consistent with the heuristic that non-abelian nilpotent structures carry negative curvature at some scale: after rescaling and passing to the blow-down, this negative contribution degenerates in the limiting sub-Riemannian geometry.

\medskip

We conclude this section by presenting one further analogy between the discrete framework of this paper and the Riemannian framework discussed above.
Let $G$ be an abstract group.
By \cite[Theorem~2.1]{Wilking} ($(2)\Rightarrow (1)$ is proved in \cite{CheegerGromollPi1,Cheeger-Gromoll-splitting}), the following are equivalent:
\begin{itemize}
\item[(1)] $G$ is finitely generated and virtually abelian;
\item[(2)] there exists a complete Riemannian manifold $(\widetilde M,\widetilde g)$ with non-negative Ricci curvature such that $G$ is a discrete closed subgroup of ${\rm Iso}(\widetilde M,\widetilde g)$ acting freely and cocompactly;
\item[(3)] there exists a complete Riemannian manifold $(\widetilde M,\widetilde g)$ with non-negative sectional curvature such that $G$ is a discrete closed subgroup of ${\rm Iso}(\widetilde M,\widetilde g)$ acting freely and cocompactly.
\end{itemize}
On the discrete side,  the following are equivalent:
\begin{itemize}
\item[(1')] $G$ is finitely generated and virtually abelian;
\item[(2')] there exist a group $\Gamma$ and a finite symmetric set of generators $S$ such that the Cayley graph of $\Gamma$ with respect to $S$ has non-negative Ollivier--Ricci curvature, and $G$ acts freely and cocompactly on it by graph isomorphism;
\item[(3')] there exist a group $\Gamma$ and a finite symmetric set of generators $S$ such that the Cayley graph of $\Gamma$ with respect to $S$ satisfies \eqref{efrvcrdscf}, and $G$ acts freely and cocompactly on it by graph isomorphism.
\end{itemize}
Indeed, for $(1')\Rightarrow(3')$, we take $\Gamma=G$ and $S$ given by Theorem \ref{thm:virtabel1:rest}, which is a strengthened version of Theorem \ref{thm:main}. Of course, $G$ acts on itself by left-multiplication.
$(3')\Rightarrow(2')$ is trivial. Now, take $\Gamma$ and $S$ as in $(2')$. By the main result of the forthcoming \cite{BHM} and Gromov's Theorem \cite{MR623534}, $\Gamma$ is virtually nilpotent. By Theorem \ref{thm:main}, $\Gamma$ is virtually abelian, hence quasi-isometric to $\ZZ^d$, for some $d$. As the Cayley graph of $\Gamma$ is locally finite and the action is free, the action is also proper. Hence, by  the Švarc--Milnor Lemma (e.g., \cite[Proposition 8.19]{BH99}), we see that $G$ is quasi-isometric to $\Gamma$, hence quasi-isometric to $\ZZ^d$, for some $d$. Therefore, $(1')$ follows from \cite[Theorem 1.1]{Shalom}. 

We finally remark that, in the above equivalence, Cayley graphs can be replaced by transitive graphs.
In particular, finitely generated virtually abelian groups are precisely those admitting a free and cocompact action by graph automorphisms on a connected transitive graph of non-negative Ollivier--Ricci curvature.
This is proved exactly as above.

\subsection{The use of AI}
In parts of this work, the authors were assisted by ChatGPT 5.5 Thinking.
In particular, while studying the toy example of the Heisenberg group $H_3(\mathbb{Z})$ with standard generators and uniform measures on balls, ChatGPT suggested a version of the potential defined in Subsection \ref{subsec:Kantorovich}.
It also suggested possible extensions beyond the specific setting of $H_3(\mathbb{Z})$, although the final presentation in this note differs from those suggestions.

When prompted with specific questions about the asymptotic behavior of the distance in the discrete setting, ChatGPT's answers inspired the authors and contributed to the development of Subsection \ref{subsec:auxiliary_distance}.
ChatGPT also proposed an argument for estimating the convolutions $\mu^{\ast n}$, but the authors found this argument overcomplicated and did not use it.
ChatGPT was also useful in checking the proof of Theorem \ref{thm:virtabel1:rest}, specifically in verifying that the claimed set of generators can be chosen as stated.
Finally, ChatGPT suggested the observation in Remark \ref{rempianofurbo} after the authors prompted it with the group to look for the example.

ChatGPT was also used for proofreading.
All mathematical arguments, computations, and conclusions were independently verified by the authors. No text in this article was written by AI.

\subsection{Acknowledgments}
Part of this work was carried out while CB was a Member and EB was a von Neumann Fellow at the Institute for Advanced Study; they gratefully acknowledge its excellent working conditions and support. This material is based upon work supported by the National
Science Foundation under Grant No.\ DMS-2424441.

\section{Setting}

Let $\Gamma$ be a finitely generated group and let $S\subseteq \Gamma$ be a fixed finite symmetric set of generators, that is, $S=S^{-1}$.
It is well known that $S$ induces the word-length function
\begin{equation}\label{brgfdsvcx}
|g|_S \defeq \min\{k\in \mathbb{N} : g=s_1 \cdots s_k,\ s_i\in S\} \qquad \text{for }g\in \Gamma\,,
\end{equation}
which then induces the left-invariant distance $d_S:\Gamma\times\Gamma\to\mathbb{N}$ defined by
\begin{equation}
d_S(g,h)\defeq |g^{-1}h|_S\qquad\text{for }
 g,h\in \Gamma\,.
\end{equation}

\subsection{Virtually nilpotent}\label{subsec:virtually}

In this work, we will always assume that $\Gamma$ is finitely generated and  virtually nilpotent. This means that $\Gamma$ admits a nilpotent subgroup of finite index.
We now follow the notation of \cite{Alexopoulos}, which in turn relies on \cite{Raghunathan}.
As observed there, it is standard to find a nilpotent, finitely generated, torsion-free subgroup $\Gamma_N\unlhd \Gamma$ of finite index, that is, $|\GGN|<\infty$.

We choose a section $\sigma:\GGN\to\Gamma$ of the projection $\pi:\Gamma\to\GGN$, that is, $\pi \circ\sigma=\operatorname{id}_{\GGN}$.
We also assume that $\sigma(e)=e$.
Then every element $g\in\Gamma$ can be written uniquely in the form
\begin{equation}\label{eq:g_n}
    g=g_N \sigma(\pi(g))\,.
\end{equation}
Notice that $g_N\in\Gamma_N$.
Moreover, since $\GGN$ is finite, there exists $C_\sigma>0$ such that
\begin{equation}\label{csdcdcsc}
    d_S(g,g_N)\le C_\sigma\qquad\text{for every $g\in\Gamma$}\,.
\end{equation}

We consider the isolated lower central series of $\Gamma_N$.
It is obtained by enlarging the lower central series
\begin{equation}
\gamma_1(\Gamma_N)=\Gamma_N,\qquad \gamma_{i+1}(\Gamma_N)=[\Gamma_N,\gamma_i(\Gamma_N)]\,,
\end{equation}
by setting
\begin{equation}\label{eq:square_root_gamma}
\Gamma_{N,i}\defeq\sqrt{\gamma_i(\Gamma_N)}
=\{h\in \Gamma_N : h^k\in \gamma_i(\Gamma_N) \text{ for some $k\in \mathbb{N}\setminus\{0\}$}\}\, .
\end{equation}

Notice that $\Gamma_{N,i}$ is a normal subgroup of $\Gamma$ for every $i$.
Indeed, it is a subgroup because it is the preimage of the torsion subgroup of $\sfrac{\Gamma_N}{\gamma_i(\Gamma_N)}$ through the projection map $p_i:\Gamma_N\to \sfrac{\Gamma_N}{\gamma_i(\Gamma_N)}$. Here we recall that the elements of finite order in a nilpotent group form a subgroup.
Next we observe that $\gamma_i(\Gamma_N)$ is normal in $\Gamma$. This follows by induction on $i$. For $i=1$, simply notice that $\Gamma_{N,1} = \Gamma_N\unlhd\Gamma$. Now, $\gamma_{i+1}(\Gamma_N)$ is generated by commutators of the type $[h,k]$, where $h\in \Gamma_N$ and $k\in\gamma_{i}(\Gamma_N)$. Hence, 
\begin{equation}
    g[h,k]^{\pm 1} g^{-1}=[g h g^{-1}, gkg^{-1}]^{\pm 1}\in\gamma_{i+1}(\Gamma_N)\qquad\text{for every }g\in\Gamma\,,
\end{equation}
as $\Gamma_N\unlhd \Gamma$ and $\gamma_i(\Gamma_N)\unlhd \Gamma$ by the inductive assumption. It then follows that $\gamma_{i+1}(\Gamma_N)\unlhd\Gamma$.

Finally, if $h\in\Gamma_{N,i}$ and $g\in\Gamma$, then $(ghg^{-1})^k=g h^k g^{-1}\in\gamma_i(\Gamma_N)$ for some $k\ge1$, and therefore $ghg^{-1}\in\Gamma_{N,i}$.

\medskip
For every $i\ge 1$, we have the lower central series property
\begin{equation}\label{cdscscd}
    [\Gamma_{N},\Gamma_{N,i}]\subseteq\Gamma_{N,i+1}\,.
\end{equation}
Indeed, let $h\in \Gamma_{N,i}$, so that $h^k\in \gamma_i(\Gamma_N)$ for some $k\in\mathbb{N}\setminus\{0\}$.
Then $[\Gamma_N,h^k]\subseteq \gamma_{i+1}(\Gamma_N)\subseteq\Gamma_{N,i+1}$, and hence $q_{i+1}(h^k)\in Z\big(\sfrac{\Gamma_N}{\Gamma_{N,i+1}}\big)$ where $q_i: \Gamma_N \to \sfrac{\Gamma_N}{\Gamma_{N,i}}$ is the quotient map.
Since $\sfrac{\Gamma_N}{\Gamma_{N,i+1}}$ is nilpotent and torsion-free, we have $p_i(h)\in Z\big(\sfrac{\Gamma_N}{\Gamma_{N,i+1}}\big)$.
Equivalently, $[\Gamma_N,h]\subseteq \Gamma_{N,i+1}$.

\medskip
If $c\in\mathbb{N}$ denotes the nilpotency class of $\Gamma_N$, then $\gamma_{c+1}(\Gamma_N)=\{e\}$.
Thus $\Gamma_{N,c+1}=\{e\}$, since $\Gamma_N$ is torsion-free.
In particular,
\begin{equation}\label{isolcent}
\Gamma_N=\Gamma_{N,1}\unrhd\Gamma_{N,2}\unrhd\cdots \unrhd\Gamma_{N,c+1}=\{e\}\,.
\end{equation}
By \eqref{cdscscd} and the definition of $\Gamma_{N,i}$ in \eqref{eq:square_root_gamma}, we have
\begin{equation}
A_i\defeq\sfrac{\Gamma_{N,i}}{\Gamma_{N,i+1}}\cong\mathbb{Z}^{n_i}\qquad\text{for every $i=1,\dots,c$}\,.
\end{equation}
Notice that our notation for the integers $n_i$ differs from that of \cite{Alexopoulos}.
We denote by
\begin{equation}
\pi_i:\Gamma_{N,i}\to A_i
\end{equation}
the quotient projections.
In what follows, it will be convenient to extend $A_i$ to the real vector space $A_i\otimes_{\mathbb{Z}}\mathbb{R}\cong\mathbb{R}^{n_i}$ and to consider the composition of the projection map with the inclusion into this vector space.
With a slight abuse of notation, we will denote this map by
\begin{equation}
\pi_i:\Gamma_{N,i}\to A_i\otimes_{\mathbb{Z}}\mathbb{R}\,.
\end{equation}

\subsection{Markov kernels}\label{sect:mk}
Let $\mu$ be a probability measure on $\Gamma$ whose support $\supp(\mu)$ contains $e$, is bounded and generates $\Gamma$.
We will also assume that $\mu$ is symmetric, i.e., 
\begin{equation}
    \mu(g)=\mu(g^{-1}) \quad \text{for every $g\in\Gamma$}\,.
\end{equation}

We often think of $\mu$ as the law of a Markov process.
The associated kernel at $g\in\Gamma$ is given by the left-translation
\begin{equation}
\mu_g(h)\defeq (L_g)_*\mu(h)=\mu(g^{-1}h)\qquad\text{for }
 h\in\Gamma\,.
\end{equation}
With this interpretation, the $n$-fold convolution $\mu^{\ast n}$ represents the law of the process at time $t=n$, and the kernel at $g\in\Gamma$ is denoted by
\begin{equation}
\mu^{\ast n}_g\defeq (L_g)_*\mu^{\ast n}\,.
\end{equation}
Recall that the convolution of two measures is defined by
\begin{equation}
\mu\ast\nu(g)
\defeq
\sum_{h\in\Gamma}\mu(h)\nu(h^{-1}g)
\qquad \text{for }g\in \Gamma\,.
\end{equation}

Let now $D\in\NN$ be the growth exponent of $\Gamma_N$, see \cite{Bass}, i.e.,
\begin{equation}\label{VolumeDoubling}
    C^{-1}n^D\le |B_n(e)|\le C n^D\qquad\text{for every $n\ge 1$}\,,
\end{equation}
where $B_n^{d_S}(e)$ can be equivalently taken in $\Gamma_N$ or $\Gamma$.
We recall the following Gaussian estimates of \cite[Theorem 5.1]{HebischSaloff}: for every $n\ge 1$,
\begin{equation}\label{GaussianBounds}
    \begin{split}
    \mu^{\ast n}(g)\le \frac{C}{n^{D/2}}\exp{-\frac{|g|_S^2}{Cn}}\qquad&\text{for every }g\in\Gamma\\
    \frac{1}{C{n}^{D/2}}\exp{-C\frac{|g|_S^2}{n}}\le \mu^{\ast n}(g)\qquad&\text{for every $g\in\Gamma$ with $|g|_S\le n/C$}\,.
    \end{split}
\end{equation}

Consider now the map $\Gamma\ni g\mapsto g_N\in\Gamma_N\unlhd\Gamma$ (see \eqref{eq:g_n}) and define a probability measure on $\Gamma_N$, for every $n\in\NN$, as
\begin{equation}\label{eq:nu_n}
    \nu^{(n)}\defeq (\,\cdot\,_N)_*\mu^{\ast n}\,,
\end{equation}
i.e., $\nu^{(n)}(g)=\sum_{b\in\GGN}\mu^{\ast n}(g\sigma(b))$. For $n\in\NN$ and $g\in\Gamma_N$, set 
\begin{equation}
 \nu^{(n)}_g=(L_g)_*\nu^{(n)}=(\,\cdot\,_N)_*\mu^{\ast n}_g\,.   
\end{equation}
In general, $\nu^{(n)}$ is \emph{not} symmetric and $\nu^{ (n)}\ne (\nu^{(1)})^{\ast n}$,
which explains the choice of the notation $\nu^{(n)}$.
 By \eqref{csdcdcsc}, we have
\begin{equation}\label{brfdv}
    W_1^{d_S}(\mu^{\ast n}_g,\nu^{(n)}_g)\le C_\sigma\qquad\text{for every $g\in\Gamma_N$ and $n\in\NN$}\,.
\end{equation}
Notice that the Gaussian bounds \eqref{GaussianBounds} hold also for $\nu^{ (n)}$ in place of $\mu^{\ast n}$. Moreover, there exists $C$ such that  $\supp(\nu^{(n)})\subseteq B_{Cn}^{d_S}(e)$ for every $n\ge 1$.

\begin{lemma}\label{vfedsc}
For every $H\in\NN$,
\begin{equation}
    \nu^{ (n^2)}\big(\Gamma_N\setminus B_{Hn}^{d_S}(e)\big)\le Ce^{-H^2/C}\qquad\text{for every $n\in\NN$}\,.
\end{equation}    
\end{lemma}

\begin{proof} 
For simplicity of notation, we write $B_R$ in place of $B_R^{d_S}(e)$.
Notice that by the bound on $\supp(\nu^{(n^2)})$,  we can assume $H\le Cn$, otherwise there is nothing to show.
We compute,   for $C_1>0$ large enough (to have the condition of the last line of \eqref{GaussianBounds}),
\begin{align}
  \nu^{(n^2)}(\Gamma_N\setminus B_{Hn})&\le \nu^{(n^2)}(\Gamma_N\setminus B_{n^2/C_1})+
\nu^{(n^2)}(B_{n^2/C_1}\setminus B_{Hn})
\\&\le (Cn^2)^{D}\frac{C}{n^D}\exp{-\frac{(n^2/C_1)^2}{Cn^2}}+e^{-H^2/C_1}C\nu^{( \lceil Cn^2\rceil)}(\Gamma_N)\\
&\le Cn^D e^{-n^2/C}+Ce^{-H^2/C}\le Ce^{-H^2/C}\,,
\end{align}
where we used also \eqref{VolumeDoubling} and \eqref{GaussianBounds}.
\end{proof}

\begin{lemma}\label{scacs}
Let $\alpha:A_1\otimes_\ZZ \RR\rightarrow\RR$ be linear and not identically $0$. Then, it holds that
    \begin{alignat}{2}
        \sum_{g\in\Gamma_N}\alpha^2(\pi_1(g))\nu^{(n)}(g)&\ge n /C&\qquad\text{for every $n\ge C$}\,,\label{revdscc0}\\
        \sum_{g\in\Gamma_N}\alpha^4(\pi_1(g)) \nu^{(n)}(g)&\le C n^2&\qquad\text{for every $n\ge C$} \,\label{revdscc01},
    \end{alignat}where $C$ depends also on $\alpha$.
\end{lemma}
\begin{proof}
To simplify the notation, we write $\alpha$ in place of $\alpha\circ\pi_1$.
We  start from exact identities for a symmetric probability measure $\lambda$ on $\Gamma_N$ with bounded support:  for every $n\in\NN$,

\begin{align}
        \sum_{g\in\Gamma_N}\alpha^2(g)\lambda^{\ast n}(g)&=n \sum_{g\in\Gamma_N}\alpha^2(g)\lambda(g)\,\label{revdscc}\\
        \sum_{g\in\Gamma_N}\alpha^4(g) \lambda^{\ast n}(g)&=n\sum_{g\in\Gamma_N}\alpha^4(g)\lambda(g)+3n(n-1)\Big(\sum_{g\in\Gamma_N}\alpha^2(g)\lambda(g)\Big)^2\label{revdscc1}\,.
    \end{align}
    Indeed, if $\gamma$ and $\varphi$ are symmetric probability measures on $\Gamma_N$ with bounded support, then
    \begin{align}
        \sum_{g\in\Gamma_N} \alpha^2(g)\gamma\ast\varphi(g)&=\sum_{g\in\Gamma_N}\alpha^2(g)\sum_{h\in\Gamma_N}\gamma(h)\varphi(h^{-1}g)=\sum_{h\in\Gamma_N}\sum_{g\in\Gamma_N}\alpha^2(g)\gamma(h)\varphi(h^{-1}g)\\
    &=\sum_{h\in\Gamma_N}\sum_{g\in\Gamma_N}\alpha^2(hg)\gamma(h)\varphi(g)=\sum_{h\in\Gamma_N}\sum_{g\in\Gamma_N}(\alpha(h)+\alpha(g))^2\gamma(h)\varphi(g)\\
    &=\sum_{h\in\Gamma_N}\sum_{g\in\Gamma_N}\alpha^2(h)\gamma(h)\varphi(g)+\sum_{h\in\Gamma_N}\sum_{g\in\Gamma_N}\alpha^2(g)\gamma(h)\varphi(g)\\
    &\qquad\qquad+2\sum_{h\in\Gamma_N}\sum_{g\in\Gamma_N}\alpha(h)\alpha(g)\gamma(h)\varphi(g)\,.
    \end{align}
    Therefore, since $\gamma$ and $\varphi$ are symmetric probability measures with bounded support
\begin{equation}
    \sum_{g\in\Gamma_N} \alpha^2(g)\gamma\ast\varphi(g)=\sum_{h\in\Gamma_N}\alpha^2(h)\gamma(h)+\sum_{g\in\Gamma_N}\alpha^2(g)\varphi(g)\,,
\end{equation}
so that \eqref{revdscc} follows by recursion. Similarly, \eqref{revdscc1} follows by computing 
\begin{align}
        \sum_{g\in\Gamma_N} \alpha^4(g)\gamma\ast\varphi(g)&=\sum_{h\in\Gamma_N}\alpha^4(h)\gamma(h)+\sum_{g\in\Gamma_N}\alpha^4(g)\varphi(g)+6\sum_{h\in\Gamma_N}\alpha^2(h)\gamma(h)\sum_{g\in\Gamma_N}\alpha^2(g)\varphi(g)\,.
\end{align}

Now we conclude the proof of the lemma. Take as $\lambda$ a  symmetric probability measure $\lambda$ on $\Gamma_N$ such that $\supp(\lambda)$ contains $e$, is bounded and generates $\Gamma_N$. For example, $\lambda$ can be the  $1/2$-lazy symmetric random walk associated to a symmetric  set of generators of $\Gamma_N$. Notice that $\lambda^{\ast n}$ satisfies the Gaussian bounds \eqref{GaussianBounds} by \cite[Theorem 5.1]{HebischSaloff}. We now take  $C_1>0$ large enough (to have the condition of the last line of \eqref{GaussianBounds}), and we record  that, for every $n\ge 1$,
\begin{equation}
        C^{-1} \lambda^{\ast \lfloor n/C\rfloor}(g)\le\nu^{(n)}(g)\le C\lambda^{\ast \lceil Cn \rceil}(g)\qquad\text{for every $g\in\Gamma_N$ with $|g|_S\le n/C_1$}\,,
\end{equation}
which follows from \eqref{GaussianBounds} applied to both $\lambda^{\ast n}$ and $\nu^{(n)}$. Hence, abbreviating   $B_R^{d_S}(e)$ into $B_R$, using \eqref{revdscc},
\begin{align}
     \sum_{g\in\Gamma_N}\alpha^2(g)\nu^{(n)}(g)&\ge   C^{-1} \sum_{g\in B_{n/C_1}}\alpha^2(g)\lambda^{\ast \lfloor n/C\rfloor}(g)\\&=C^{-1}\sum_{g\in \Gamma_N}\alpha^2(g)\lambda^{\ast \lfloor n/C\rfloor}(g)-C\sum_{g\in \Gamma_N\setminus B_{n/C_1}}\alpha^2(g)\lambda^{\ast \lfloor n/C\rfloor}(g)\\
     &\ge C^{-1}\lfloor n/C\rfloor-(Cn)^2\lambda^{\ast \lfloor n/C\rfloor}(\Gamma_N\setminus B_{n/C_1})\\
     &\ge n/C -Cn^2e^{-n/C}\,
\end{align}
where the bound on $\lambda^{\ast \lfloor n/C\rfloor}(\Gamma_N\setminus B_{n/C_1})$ used in the last inequality is as for Lemma \ref{vfedsc}. Hence, \eqref{revdscc0} follows. The argument for \eqref{revdscc01} is analogous and hence omitted. 
\end{proof}

\subsection{Quotient spaces}
\label{subsec:quotientspaces}

In the proof of our main Theorem \ref{thm:main}, it is convenient, though not strictly necessary, to reduce to the case of virtually nilpotent groups of step two.
The standard way to do this is by quotienting out a suitable term of the isolated lower central series.
To this end, we need to study the induced structure on the quotient.
We do this below in a fairly general setting.

\medskip
Let $\Gamma$ be a finitely generated group, not necessarily virtually nilpotent, with symmetric generating set $S$.
Let $K\unlhd \Gamma$, and denote by $p:\Gamma\to\sfrac{\Gamma}{K}$ the quotient projection.
We often use the shorthand notation $\overline g\defeq p(g)$.

Notice that $\overline S\defeq p(S)$ is a symmetric generating set of $\sfrac{\Gamma}{K}$, which induces the word-length function $|\,\cdot\,|_{\overline S}$ and the associated word distance.
It is an easy exercise to show that this word distance coincides with the orbit distance on the quotient. Indeed, we have
\begin{equation}\label{eq:orbit_quotient}
|\overline g|_{\overline S}=\min_{k\in K}|gk|_S
\qquad \text{for every }g\in \Gamma\,,
\end{equation}
which, together with the normality of $K$ and the left-invariance of $d_S$, implies
\begin{equation}\label{eq:orbit_distance}
d_{\overline S}(\overline g,\overline h)=\min_{k_1,k_2\in K} d_S(gk_1,hk_2)
\qquad \text{for every }g,h\in \Gamma\, .
\end{equation}

Let $\mu$ be a symmetric probability measure on $\Gamma$ whose support is finite, generates $\Gamma$, and contains $e$.
We define $\overline\mu\defeq p_*\mu$, that is,
\begin{equation}
\overline\mu(\overline g)\defeq \sum_{k\in K}\mu(gk)\qquad \text{for }
\overline g\in \sfrac{\Gamma}{K}\,.
\end{equation}
Notice that $\overline\mu$ is a symmetric probability measure on $\sfrac{\Gamma}{K}$ whose support is finite, generates $\sfrac{\Gamma}{K}$, and contains the identity element. Then we set
\begin{equation}
   \overline \mu^{\ast n}_{\overline{g}}(\overline h)
   \defeq \overline \mu^{\ast n}(\overline g^{-1}\overline h)
   \qquad\text{for } 
   \overline g, \overline h \in \sfrac{\Gamma}{K}, \, n\in \NN\,.
\end{equation}

\begin{lemma}\label{dfcsc}
For every $g\in \Gamma$, it holds
    \begin{equation}
        W_1^{d_S}(\mu^{\ast n}_g,\mu^{\ast n})
        \ge W_1^{d_{\bar S}}(\overline \mu^{\ast n}_{\overline{g}},\overline \mu^{\ast n})\qquad\text{for every $n\in\NN$}\,.
    \end{equation}
\end{lemma}
\begin{proof}
Notice first that 
\begin{equation}
    \overline \mu^{\ast n}_{\overline g}=p_*\mu^{\ast n}_g\qquad\text{for every $n\in\NN$}\,,
\end{equation} where we also used that  $p$ is a homomorphism, as $K$ is normal. 
Also, for every pair of probability measures $\alpha,\beta$ on $\Gamma$ with bounded support,
\begin{equation}
W_1^{d_S}(\alpha,\beta)\ge W_1^{d_{\overline S}}(p_* \alpha,p_* \beta)\,,
\end{equation}
since the pushforward of an admissible plan for $W_1^{d_S}(\alpha,\beta)$ on $\Gamma$ is an admissible plan for $W_1^{d_{\overline S}}(p_*\alpha,p_*\beta)$ on $\sfrac{\Gamma}{K}$, and the quotient distance satisfies \eqref{eq:orbit_distance}. 
\end{proof}

\section{Main results restated}

We state more precise and general versions of Theorems \ref{thm:main}, \ref{vecdscs}, and \ref{cdscs}, which were presented in the introduction.
We then explain how the latter two imply the first one.

\begin{theorem}[General version of Theorem \ref{thm:main}]\label{thm:main:rest}
Let $\Gamma$ be a finitely generated virtually nilpotent group. Then, the following are equivalent
\begin{enumerate}
    \item $\Gamma$ is virtually abelian, 
    \item There exists a finite symmetric set of generators $S$ with $e\notin S$ such that 
    \begin{equation}
    \kappa(g,h)= 1-\frac{W_1^{d_S}(\mu_g,\mu_h)}{d_S(g,h)}\ge 0\qquad\text{for every $g\sim h$}\,,
    \end{equation}
    where $\mu_g$ is the law of the $1/2$-lazy symmetric random walk originating from $g$ with respect to $S$. This is to say that the Cayley graph of $\Gamma$ with respect to $S$ and the $1/2$-lazy symmetric  random walk associated to $S$ has non-negative Ollivier--Ricci curvature.
    \item There exists a finite symmetric set of generators $S$ and a symmetric probability measure $\mu$ such that $\supp(\mu)$ contains $e$, is bounded and generates $\Gamma$, such that
    \begin{equation}
    \kappa(g,h)= 1-\frac{W_1^{d_S}(\mu_g,\mu_h)}{d_S(g,h)}\ge 0\qquad\text{for every  $g\sim h$}\,,
    \end{equation}  
    where $\mu_g= (L_g)_*\mu$. This is to say that the Cayley graph of $\Gamma$ with respect to $S$ and $\mu$ has non-negative Ollivier--Ricci curvature.
\end{enumerate}
\end{theorem}
\begin{theorem}[General version of Theorem \ref{vecdscs}]\label{vecdscs:res}
Let $\Gamma$ be a finitely generated virtually nilpotent group. Let $S\subseteq\Gamma$ be a finite symmetric set of generators,  and let $\mu$  be a symmetric probability measure such that $\supp(\mu)$ contains $e$, is bounded and generates $\Gamma$.
Assume that $\Gamma$  is  not virtually abelian.  Then, there exists a sequence $(g_n)_n\subseteq\Gamma$ with 
    \begin{equation}
        \lim_{n\rightarrow\infty}\frac{|g_n|_S}{n}\in (0,\infty)
    \end{equation}such that
    \begin{equation}
        \liminf_{n\rightarrow\infty} \frac{W_1^{d_S}(\mu^{\ast n^2}_{g_n},\mu^{\ast n^2}_e)}{|g_n|_S}>1\,.
    \end{equation}
\end{theorem}
\begin{theorem}[General version of Theorem  \ref{cdscs}]\label{cdscs:rest}
     Let $\Gamma$ be a finitely generated virtually abelian group. Let $S\subseteq\Gamma$ be a finite symmetric set of generators,  and let $\mu$  be a symmetric probability measure such that $\supp(\mu)$ contains $e$, is bounded and generates $\Gamma$. Then,
    \begin{equation}
        \big|W_1^{d_S}(\mu^{\ast n}_e,\mu^{\ast n}_g)- |g|_S\big|\le C\qquad\text{for every $n\in\NN$ and $g\in\Gamma$}\,,
    \end{equation}
    where $C$ depends only on $\Gamma$, $S$ and $\mu$.
\end{theorem}
\begin{remark}\label{rempianofurbo}

   One may wonder whether, in the proof of Theorem \ref{cdscs:rest}, the trivial transport map $h\mapsto gh$ is sufficient to estimate the Wasserstein distance, since $\Gamma$ is virtually abelian.
Surprisingly, this is already false for the infinite dihedral group of Example \ref{infdie}: even at large scales, this transport map is far from optimal.
We refer to Subsection \ref{subsec:proof_remark} for the detailed statement and proof.
\end{remark}

In view of the following Theorem, we recall the definition of $\infty$-Wasserstein distance with respect to $d_S$, i.e., for probability measures $\mu,\nu$,
\begin{equation}
    W_\infty^{d_S}(\mu,\nu)\defeq \inf_{\pi} \sup_{\pi(g,h)>0} d_S(g,h)\,,
\end{equation}
where the infimum is taken among all transport plans $\pi$ for $(\mu,\nu)$. As the transport plans $\pi$ are probability measures, it follows immediately that $W_1^{d_S}\le W_\infty^{d_S}$. We remark also that in \cite[Problem P]{ollivier2010survey}, Ollivier proposed the notion of sectional curvature by using $W_\infty$ instead of $W_1$ in the definition of $\kappa$. Hence, in this language, \eqref{efrvcrdscf}  below means ``non-negative Ollivier--Sectional curvature".
\begin{theorem}\label{thm:virtabel1:rest}
    Let $\Gamma$ be a finitely generated virtually abelian group. Then, there exists a finite symmetric set of generators $S$ with  $e\notin S$ such that, if $ \mu$ denotes the $1/2$-lazy symmetric random walk associated to $ S$, it holds that 
    \begin{equation}\label{efrvcrdscf}
        W_\infty^{d_{ S}}(\mu^{\ast n}_e,\mu^{\ast n}_g)\le |g|_S\qquad\text{for every $n\in\NN$ and $g\in\Gamma$}\,.
    \end{equation}
\end{theorem}

\subsection{Proof of Theorem \ref{thm:main:rest} given Theorems \ref{vecdscs:res} and \ref{thm:virtabel1:rest}}
$(1)\Rightarrow (2)$ follows from Theorem~\ref{thm:virtabel1:rest}, as $W_1^{d_S}(\mu^{\ast n}_e,\mu^{\ast n}_g)\le W_\infty^{d_S}(\mu^{\ast n}_e,\mu^{\ast n}_g)$. $(2)\Rightarrow (3)$ is clear. $(3)\Rightarrow(1)$ follows from Theorem~\ref{vecdscs:res} with Remark \ref{rem:vfdsc}. Notice that, even though Remark \ref{rem:vfdsc} was originally introduced only for the laws of the $1/2$-lazy symmetric random walks associated to $S$, the argument carries over  \textit{verbatim} for any measure $\mu$ as the ones considered.
\qed
\subsection{Transitive graphs}
 As we did for Cayley graphs, we state the general versions of our results for Markov kernels instead of random walks. As discussed before, this is an inessential difference. For Cayley graphs we  used the notation $\mu$ for kernels. Here, we use the standard notation $P$ for  transition matrices. Notice that, for a Cayley graph, a Markov kernel as in the statement of Theorem \ref{notationsimplify}  below corresponds to a kernel $\mu$. As in Subsection \ref{veadscsdc}, we naturally have  a notion of Ollivier--Ricci curvature associated to these Markov kernels,  namely
 \begin{equation}
     \kappa(x,y)\defeq 1-\frac{W_1(P(x,\,\cdot\,),P(y,\,\cdot\,))}{d(x,y)}\qquad\text{for }x\ne y\,.
 \end{equation}
    \begin{theorem}[General version of Theorem \ref{thm:verttrans}]
\label{notationsimplify}
	Let $G=(V,E)$ be a connected graph of polynomial growth and assume that we have $\Gamma\le \Aut(G)$ which acts transitively. Let $P(\,\cdot\,,\,\cdot\,):V\times V\rightarrow[0,1]$ be an irreducible symmetric Markov kernel which is diagonally invariant for $\Gamma$, is lazy and is compactly supported, more precisely,
	\begin{align}
		&P( x, y)=P(y,x)\qquad&&\text{for every $x,y\in V$}\,,\\
		&P(\gamma x,\gamma y)=P(x,y)\qquad&&\text{for every $\gamma\in\Gamma$ and $x,y\in V$}\,,\\
		&P(x,x)>0\qquad&&\text{for one (and hence any) $x\in V$}\,,\\
		&\{y:P(x,y)>0\}\text{ is finite}\qquad&&\text{for one (and hence any) $x\in V$}\,.
	\end{align}
	
	Assume  that $G$ has non-negative Ollivier--Ricci curvature with respect to $P$. Then $G$ is quasi-isometric to $\mathbb{Z}^k$ for some $k\in\mathbb{N}$.
\end{theorem}
\section{The polytope}
\label{sec:poly}

We follow the notation introduced in Section \ref{subsec:virtually}.
We endow $A_1\otimes_{\mathbb{Z}}\mathbb{R}$ with its unique Hausdorff vector space topology.
We consider
\begin{equation}\label{csdccdsc} E\defeq\overline{\operatorname{co}}\bigg\{\frac{\pi_1(g)}{|g|_S}:g\in \Gamma_N\setminus\{e\}\bigg\}\,, \end{equation}
where $\overline{\operatorname{co}}$ denotes the closed convex hull in the real vector space $A_1\otimes_{\mathbb{Z}}\mathbb{R}$. We will denote by $\Ext(E)$ the set of extreme points of $E$. The following fact seems to follow from a rather standard argument, though we could not find  clear references. Hence we give a detailed proof for the sake of clarity and completeness.

\begin{lemma}\label{cdescdsacsc}
The set $E$ is convex, compact, symmetric, and has nonempty interior.
Moreover, $E$ has finitely many extreme points, all of which belong to $A_1\otimes_\ZZ\QQ$.
\end{lemma}

\begin{proof}
The set $E$ is convex, symmetric, and has nonempty interior, since it contains suitably scaled copies of the generators of $A_1$.
We provide a characterization of $E$ which will make clear that it is compact and has finitely many extreme points contained in $A_1\otimes_\ZZ \QQ$.

\medskip

For the sake of exposition, we first consider the simpler case where $\Gamma=\Gamma_N$.
We prove that $E$ coincides with the convex envelope of generators:
\begin{equation}
E'\defeq \operatorname{co}\{\pi_1(s):s\in S\}\,,
\end{equation}
which is clearly compact and $\Ext(E')$ is finite and contained in $A_1\subseteq A_1\otimes_\ZZ \QQ$.

Since $|s|_S=1$ for every $s\in S\setminus \{e\}$, clearly $\pi_1(s)= \frac{\pi_1(s)}{|s|_S}\in E$.
Hence $E'\subseteq E$.
To prove that $E\subseteq E'$, we write any $g\in \Gamma_N\setminus \{e\}$ as $g=s_1\cdots s_k$, where $s_i\in S$ and $|g|_S=k$.
Hence
\begin{equation}
    \frac{\pi_1(g)}{|g|_S}
    = \frac{\pi_1(s_1\cdots s_k)}{k}
    = \frac{1}{k} \sum_{i=1}^k \pi_1(s_i) \in E'\,.
\end{equation}
Taking closed convex hulls gives $E\subseteq E'$.

\medskip
The proof of the general case $\Gamma_N\subseteq\Gamma$ requires some more work, since in general the generators $s_i$ do not belong to $\Gamma_N$.
It turns out that the right replacements are products of generators forming simple cycles in the graph whose vertices are the elements $b\in \GGN$ of the quotient space and whose edges are $(b,b\pi(s))_{b\in\GGN,\ s\in S}$.
Notice that this graph may have self-loops and multiple edges connecting two vertices.
More precisely, we say that $(s_1,\dots,s_k)\in S^k$ is a simple cycle based at $\pi(e)$ if the vertices $\pi(e),\pi(s_1),\pi(s_1s_2),\dots,\pi(s_1\cdots s_k)$ form a simple cycle in the graph described above, with initial and terminal vertex equal to $\pi(e)$ and with no other repetitions.
Equivalently, $\pi(s_1\cdots s_k)=\pi(e)$, and the vertices $\pi(e),\pi(s_1),\dots,\pi(s_1\cdots s_{k-1})$ are pairwise distinct.

We notice crucially that simple cycles are finitely many  as $\GGN$ is finite and we consider
\begin{equation}\label{efdsccsd}
    E'\defeq\operatorname{co}\Big\{\frac{\pi_1(\sigma(b)s_1\cdots s_k\sigma(b)^{-1})}{k}: (s_1,\dots,s_k)\text{ is a simple cycle at $\pi(e)$ and }b\in\GGN\Big\}\,
\end{equation}
where $\sigma:\GGN\to\Gamma$ is a section of the projection $\pi:\Gamma\to\GGN$ as introduced in Section \ref{subsec:virtually}.
Notice that $E'$ is compact. To conclude the proof it suffices to show that $E=E'$.

To show $E'\subseteq E$, take  $(s_1,\dots,s_k)$ to be a simple cycle at $e$ and $b\in\GGN$, we have to show that 
\begin{equation}\label{edscc}
    \frac{\pi_1(\sigma(b)s_1\cdots s_k\sigma(b)^{-1})}{k}\in E\,.
\end{equation}
Of course, we can assume that $s_1\cdots s_k\ne e$. We consider, for $h\in\NN\setminus\{0\}$,  
\begin{equation}
    \Gamma_N\setminus\{e\}\ni (\sigma(b)s_1\cdots s_k\sigma(b)^{-1})^h= \sigma(b)(s_1\cdots s_k)^h\sigma(b)^{-1}\,,
\end{equation}
so that, recalling \eqref{csdcdcsc}, $|(\sigma(b)s_1\cdots s_k\sigma(b)^{-1})^h|_S\le kh+2C_\sigma$. Hence,
\begin{equation}
   E\ni  \frac{\pi_1\big((\sigma(b)s_1\cdots s_k\sigma(b)^{-1})^h\big)}{kh+2C_\sigma}=\frac{hk}{kh+2C_\sigma}\frac{\pi_1(\sigma(b)s_1\cdots s_k\sigma(b)^{-1})}{k}\,
\end{equation}
so that \eqref{edscc} follows by letting $h\rightarrow\infty$ and using that $E$ is closed.

We show now $E\subseteq E'$. Let $g\in\Gamma_N\setminus\{e\}$ and set $k\defeq |g|_S$, so that $g=s_1\cdots s_k$ for some $s_1,\dots,s_k\in S$. Now we want to decompose the cycle $(s_1,\dots,s_k)$ into simple cycles (not necessarily based at $\pi(e)$  -- notice however that every simple cycle corresponds to a simple cycle based at $\pi(e)$ through conjugation). The procedure is as follows. Assume that $(s_1,\dots,s_k)$ is not simple, then there exist $1\le i\le j\le k$ such that $(s_i,\dots, s_j)$ is a simple cycle. We then write $b\defeq \pi(s_1\cdots s_{i-1})$, and we notice that 
\begin{equation}
    s_1\cdots s_k= \underbrace{s_1\cdots s_{i-1}\sigma(b)^{-1}}_{\in\Gamma_N}\underbrace{\sigma(b) s_i\cdots s_j \sigma(b)^{-1}}_{\in\Gamma_N}\underbrace{\sigma(b)s_{j+1}\cdots s_k}_{\in\Gamma_N}.
\end{equation}
Hence, since $A\otimes_\ZZ\RR$ is abelian, using additive notation, we have
\begin{align}
    \pi_1(s_1\cdots s_k)
    &=
    \pi_1(s_1\cdots s_{i-1}\sigma(b)^{-1})
    +\pi_1(\sigma(b)s_i\cdots s_j\sigma(b)^{-1})
    +\pi_1(\sigma(b)s_{j+1}\cdots s_k)
    \\
    &=
    \pi_1(s_1\cdots s_{i-1}\sigma(b)^{-1})
    +\pi_1(\sigma(b)s_{j+1}\cdots s_k)
    +\pi_1(\sigma(b)s_i\cdots s_j\sigma(b)^{-1})
    \\
    &=
    \pi_1(s_1\cdots s_{i-1}s_{j+1}\cdots s_k)
    +\pi_1(\sigma(b)s_i\cdots s_j\sigma(b)^{-1})\,.
\end{align}
Continuing in this fashion, we see that we can write
\begin{equation}
    \pi_1(s_1\cdots s_k)
    =
    \sum_p \pi_1(\sigma(b_p)c_p\sigma(b_p)^{-1})\,.
\end{equation}
where $(s_{i_1},\dots, s_{i_{l_p}})$ is a simple cycle based at $e$, $c_p \defeq s_{i_1}\cdots s_{i_{l_p}}$, $b_p\in\GGN$ and $\sum_p {l_p=k}$, where $p$ ranges in a finite index set. Thus,
\begin{equation}
    \frac{\pi_1(g)}{|g|_S}=\frac{\pi_1(s_1\cdots s_k)}{k}=\sum_p \frac{l_p}{k}\frac{\pi_1(\sigma(b_p)c_p\sigma(b_p)^{-1})}{l_p}\in E'\, ,
\end{equation}
which concludes the proof.
\end{proof}

By Lemma \ref{cdescdsacsc}, there exists a well-defined norm $|\,\cdot\,|_E$ on the real vector space $A_1\otimes_\ZZ\RR$ associated with $E$.
Moreover, by \eqref{csdccdsc}, we have the inequality
\begin{equation}\label{eq:E-S}
|\pi_1(g)|_E \le |g|_S
\qquad \text{for every }g\in \Gamma_N\,.
\end{equation}

As a consequence of the symmetry of $E$ and of the characterization of $E$ as the convex hull of cycles in \eqref{efdsccsd}, we can characterize its extreme points as
\begin{equation}\label{vefdsc} \Ext(E)=\Big\{\frac{\pi_1(t_i^{\pm 1})}{l_i}:i= 0,\dots, m\Big\}\,, \end{equation}
where $t_0,\dots,t_m\in\Gamma_N$ are cycles of lengths $l_0,\dots,l_m$.
We assume that  $t_i^{\pm 1}\ne t_j$ for $i\ne j$, in order to avoid redundancies in the definition.
Finally, we record the bound
\begin{equation}  \label{vfsdscsx}
|t_i^k|_S\le |k| l_i + 2C_\sigma\qquad\text{for every $k\in\ZZ$ and $i=0,\dots,m$}\,,
\end{equation}
which follows directly from the cycle structure, as observed in the proof of Lemma~\ref{cdescdsacsc}.

\subsection{Special vertex}
\label{subsec:special_vertex}
We assume that $\Gamma$ is virtually nilpotent and \emph{not} virtually abelian.
Equivalently, $\Gamma$ contains a nilpotent subgroup of finite index, but no abelian subgroup of finite index.

We notice that relation \eqref{cdscscd} implies that the composition
\begin{equation}\label{eq:commutator1}
\Gamma_N\times \Gamma_N
\xrightarrow{[\,\cdot\,,\,\cdot\,]}
\Gamma_{N,2}
\xrightarrow{\pi_2}
A_2
\end{equation}
is well defined and induces a $\mathbb{Z}$-bilinear form
\begin{equation}\label{dscsc}
[\,\cdot\,,\,\cdot\,]_{\Gamma_N}:A_1\times A_1\rightarrow A_2\,,
\end{equation}
where we add the subscript $\Gamma_N$ to avoid confusion with the vanishing commutator in $A_1$.
We extend it to a $\RR$-bilinear map
\begin{equation}\label{dscsc1}
[\,\cdot\,,\,\cdot\,]_{\Gamma_N}:(A_1\otimes_\ZZ\RR)\times (A_1\otimes_\ZZ\RR)\rightarrow A_2 \otimes_\mathbb{Z} \RR\,,
\end{equation}
which we still denote with the same symbol.

\begin{lemma}\label{dscc}
Assume that $\Gamma$ is virtually nilpotent and not virtually abelian.
There exists $v_0\in\Ext(E)$ such that
    \begin{equation}
[v_0,\,\cdot\,]_{\Gamma_N}:A_1\otimes_\ZZ\RR \to A_2\otimes_\ZZ\RR
\qquad\text{is not identically zero}\,.
    \end{equation}
\end{lemma}

\begin{proof}
By Lemma \ref{cdescdsacsc}, it is enough to show that the map in \eqref{eq:commutator1} is not identically $e$.
Suppose, by contradiction, that this map is identically $e$.
Then 
\begin{equation}
    \gamma_2(\Gamma_N)=[\Gamma_N,\Gamma_N]\subseteq\Gamma_{N,3}=\sqrt{\gamma_3(\Gamma_N)}\,,
\end{equation}
which implies $\Gamma_{N,2}=\Gamma_{N,3}$.
This would imply $\Gamma_{N,2}=\{e\}$, since the isolated lower central series is strictly decreasing until it terminates.
Indeed, assume that, for some $i\ge 3$,
\begin{equation}\label{edsccsc}
\Gamma_{N,2}=\Gamma_{N,3}=\cdots = \Gamma_{N,i}\,.
\end{equation}
Then, by \eqref{cdscscd},
\begin{equation}
\gamma_i(\Gamma_N)=[\Gamma_N,\gamma_{i-1}(\Gamma_N)]\subseteq [\Gamma_N,\Gamma_{N,i-1}]
=[\Gamma_N,\Gamma_{N,i}]
\subseteq \Gamma_{N,i+1}\,,
\end{equation}
and hence $\Gamma_{N,i}\subseteq \Gamma_{N,i+1}$.
Since the converse inclusion always holds, \eqref{edsccsc} also holds with $i+1$ in place of $i$.
\end{proof}

\section{Norm and blowdown}\label{sect:noemandb}

In this section, we describe known results on the structure at infinity of the Cayley graph of a finitely generated virtually nilpotent group.
Technically, we do not use this description in the proof of the main results of this paper.
However, the discussion is useful for building intuition and for explaining some of the technical choices made in the actual proof. 

\subsection{Carnot--Carathéodory spaces}

We use the notation of Section \ref{subsec:virtually}.
Let $\Gamma$ be a finitely generated virtually nilpotent group, endowed with the word metric $d_S$ induced by a finite symmetric set of generators $S$.
By \cite{Pansu89,Breuillard}, the blow-down of the metric space $(\Gamma,d_S)$ is a stratified nilpotent Lie group $(N_\infty,d_{\rm CC})$ endowed with a Carnot--Carathéodory distance.
The blow-down can be understood in the pointed Gromov--Hausdorff sense: the family of rescaled metric spaces $(\Gamma,r^{-1}d_S,e)$ converges, as $r\to\infty$, to the metric space $(N_\infty,d_{\rm CC},e)$.
The structure of the blow-down can be described as follows: the corresponding graded Lie algebra 
\begin{equation}
\mathfrak{n}_\infty=\mathfrak{n}_1\oplus\mathfrak{n}_2\oplus\cdots\oplus\mathfrak{n}_c\,,
\end{equation}
satisfies
\begin{equation}
\mathfrak{n}_i \cong A_i\otimes_\ZZ\RR
=\sfrac{\Gamma_{N,i}}{\Gamma_{N,i+1}}\otimes_\ZZ\RR\,.
\end{equation}

The first layer $\mathfrak{n}_1$ is identified with $A_1\otimes_\ZZ\RR$.
The Carnot--Carathéodory distance $d_{\rm CC}$ is the left-invariant sub-Finsler distance obtained from the norm $|\,\cdot\,|_E$ whose unit ball is the polytope $E$ defined in Section \ref{sec:poly}.
Thus, if $x,y\in N_\infty$, then
\begin{equation}
d_{\rm CC}(x,y)
=\inf\left\lbrace
\int_0^1 |\dot\gamma(t)|_E\,dt :
\gamma(0)=x,\ \gamma(1)=y,\ \dot\gamma(t)\in \mathfrak{n}_1 \text{ for a.e. }t
\right\rbrace.
\end{equation}
Here the condition $\dot\gamma(t)\in\mathfrak{n}_1$ is understood after left translation to the identity.

\subsection{Malcev closure}
\label{sec:malcev_closure}
A concrete way to realize the Lie algebra structure $(\mathfrak{n}_\infty, [\, \cdot \, , \, \cdot ]_\infty)$ of $N_\infty$ is to consider the Malcev closure of $\Gamma_N$, namely a simply connected nilpotent Lie group $N$ containing $\Gamma_N$ as a cocompact lattice; see \cite{Raghunathan}.
The finite extension $\Gamma$ is irrelevant for the determination of the Lie algebra structure, whereas it enters decisively in the definition of the Carnot--Carathéodory distance through the polytope $E$ explained below.
To construct $N$, one typically fixes a Malcev basis of $\Gamma_N$ (see Subsection \ref{subsec:Malcev} for a construction in the step-two framework).
This gives a vector-space decomposition of the Lie algebra $\mathfrak{n}$ of $N$,
\begin{equation}\label{eq:deco_m}
\mathfrak{n}=\mathfrak{m}_1\oplus\cdots\oplus\mathfrak{m}_c
=
\mathfrak{m}_1\oplus[\mathfrak{n},\mathfrak{n}]\,.
\end{equation}
The Lie algebra structure $(\mathfrak{n},[\,\cdot\,,\,\cdot\,])$ depends on the chosen Malcev basis and does not necessarily coincide with the graded Lie algebra $(\mathfrak{n}_\infty,[\,\cdot\,,\,\cdot\,]_\infty)$, nor with its stratification.
However, the latter can be obtained through the limiting procedure
\begin{equation}\label{eq_limit_lie}
[x,y]_\infty
=
\lim_{\lambda\to\infty}
\delta_\lambda^{-1}\big[\delta_\lambda(x),\delta_\lambda(y)\big],
\qquad
x,y\in\mathfrak{n}\,,
\end{equation}
where $\delta_\lambda(x)\defeq \lambda x_1+\lambda^2 x_2+\cdots+\lambda^c x_c$ is the anisotropic scaling operator associated with the decomposition \eqref{eq:deco_m}.
The limiting procedure has the effect of retaining only the leading homogeneous component of the original bracket with respect to this decomposition.
More precisely, if $x\in\mathfrak{m}_i$ and $y\in\mathfrak{m}_j$, then $[x,y]_\infty$ is the projection of $[x,y]$ onto $\mathfrak{m}_{i+j}$, with the convention that $\mathfrak{m}_{i+j}=\{0\}$ if $i+j>c$.

Again through the fixed Malcev basis we can identify $\mathfrak{m}_1$ with $A_1 \otimes_\ZZ \RR$ and as in \cite{Breuillard}, consider the induced polytope $E$ (defined in \eqref{csdccdsc}) in $\mathfrak{m}_1$. 
Recall that by Lemma \ref{cdescdsacsc}, $E$ above is the unit ball of a suitable norm $|\,\cdot\,|_E$ on $\mathfrak{m}_1$. We then consider  $d_{\rm CC}:N\times N\rightarrow[0,\infty)$, which is the left-invariant Carnot--Carathéodory metric defined by the norm $|\,\cdot\,|_E$. By \cite[Theorem 6.2]{Breuillard}, it holds that, 
\begin{equation}
    \lim_{x\rightarrow\infty} \frac{d_S(e,x)}{d_{\rm CC}(e,x)}=1\,.
\end{equation}

\subsection{Stability of curvature bounds}

The blow-down structure $(N_\infty,d_{\rm CC})$ can be used to study the Ollivier--Ricci curvature of $(\Gamma,d_S)$ at large scales.
Concretely, one may reinterpret the statement of Theorem \ref{vecdscs} as a curvature bound on the blow-down.
Indeed, the statement that
\begin{equation}
\limsup_{n\rightarrow\infty}\kappa_{n^2}(e,g_n)<0
\end{equation}
for a sequence $(g_n)_n\subseteq\Gamma$ satisfying
$\lim_{n\rightarrow\infty}d_S(e,g_n)/n\in(0,\infty)$ suggests, after passing to a subsequence in the blow-down, the existence of a point $g\in N_\infty$ with $d_{\rm CC}(g,e)\in(0,\infty)$ such that
\begin{equation}\label{eq:limit_curvature}
\kappa_\infty(e,g)\defeq 1-\frac{W_1^{d_{\rm CC}}(\mu^\infty_g,\mu^\infty_e)}{d_{\rm CC}(g,e)}<0\,.
\end{equation}
Here $\mu^\infty_e$ denotes the scaling limit, in the Carnot--Carathéodory space $(N_\infty,d_{\rm CC})$, of the measures $\mu^{\ast n^2}_{e}$ viewed from the identity, while $\mu^\infty_g$ is its left translate by $g$.
Thus, at least formally, the uniformly negative curvature bound detected at scale $n$ for the discrete random walk should persist in the limiting sub-Finsler geometry.

Conversely, the existence of a point $g\in N_\infty\setminus\{e\}$ such that \eqref{eq:limit_curvature} holds implies the conclusion of Theorem \ref{vecdscs}. This reduction to the study of curvature on Carnot--Carathéodory spaces has both advantages and disadvantages.
The continuous asymptotic structure allows one to use analytic tools that are not available in the discrete framework, as well as a proof by rigidity (described below) which is not possible in the discrete case.
On the other hand, the argument is less self-contained, since it relies heavily on highly non-trivial results about limits at infinity, both for the metric structure and for the Markov processes.
Moreover, this strategy loses some of the combinatorial flavor and simplicity of the discrete approach.

\subsection{Possible argument in the continuous case}

Let us consider, as a toy model of a Lie group endowed with a Carnot--Carathéodory distance, the Heisenberg group
\begin{equation}\label{eq:heisenberg_R}
H_3(\RR)
=
\left\lbrace
\begin{pmatrix}
1 & a & c \\
0 & 1 & b \\
0 & 0 & 1
\end{pmatrix}
: a,b,c\in\RR
\right\rbrace.
\end{equation}
It can be thought of as the Malcev closure of $H_3(\ZZ)$ with respect to the Malcev basis 
\begin{equation}
x_1=
\begin{pmatrix}
1 & 1 & 0 \\
0 & 1 & 0 \\
0 & 0 & 1
\end{pmatrix},
\qquad
x_2=
\begin{pmatrix}
1 & 0 & 0 \\
0 & 1 & 1 \\
0 & 0 & 1
\end{pmatrix},
\qquad
y=
\begin{pmatrix}
1 & 0 & 1 \\
0 & 1 & 0 \\
0 & 0 & 1
\end{pmatrix},
\end{equation}
which satisfies the commutator relations: $[x_1,x_2]=y$ and $y$ is central.

The Lie algebra $\mathfrak{h}_3$ is the space of upper-triangular matrices with zeros along the diagonal.
It is stratified as
\begin{equation}
\mathfrak{h}_3=\mathfrak{h}_1\oplus\mathfrak{h}_2,
\qquad
\mathfrak{h}_1=\operatorname{span}\{X_1,X_2\},
\qquad
\mathfrak{h}_2=\operatorname{span}\{Y\}\,,
\end{equation}
where $X_1=x_1-I$, $X_2=x_2-I$, and $Y=y-I$.
Notice that $\exp (X_i)=x_i$, $\exp(Y)=y$, and
\begin{equation}
[X_1,X_2]=Y,
\qquad
[X_1,Y]=[X_2,Y]=0\,.
\end{equation}
A Carnot--Carathéodory distance $d_{\rm CC}$ on $H_3(\RR)$ is obtained by fixing a norm $|\, \cdot \,|$ on the horizontal layer $\mathfrak{h}_1$ and declaring admissible curves to be those whose left-translated velocity belongs to $\mathfrak{h}_1$ almost everywhere.
For the sake of simplicity, we assume that the unit ball of the norm $|\, \cdot \,|$ is the convex hull of $\pm X_1$ and $\pm X_2$.
This corresponds to the blow-down of $H_3(\ZZ)$ with the word metric generated by $S=\{x_1^{\pm1},x_2^{\pm1}\}$.

\begin{remark}
The coordinate system $(a,b,c)$ in \eqref{eq:heisenberg_R} is adapted to the normal form
\begin{equation}
    x=x_2^b x_1^a y^c\,.
\end{equation}
Thus it coincides with Malcev coordinates with respect to the ordered basis $x_2,x_1,y$, and not with respect to the natural ordering $x_1,x_2,y$.
\end{remark}

We argue that \begin{equation}
    W_1^{d_{\rm CC}}(\mu_{x_1},\mu_e) > d_{\rm CC}(x_1,e)=1
\end{equation}
where $\mu_x\defeq (L_x)_* \mu$ and $\mu$ is any Gaussian-like probability measure in $H_3(\RR)$ that contains a neighborhood of $e$ in its support.

We do not need to define the full nonlinear potential described in Section \ref{subsec:Kantorovich}, it is enough to consider the linear term and use a rigidity argument, which we remark is possible only in the asymptotic case.
With respect to the coordinates $(a,b,c)$ as in \eqref{eq:heisenberg_R}, we consider the function $\ell(x)\defeq a(x)$ for $x\in H_3(\RR)$.
Clearly, $\ell$ is $1$-Lipschitz with respect to $d_{\rm CC}$, since its horizontal differential corresponds to $X_1^*$, where we use the notation $X_1^*,X_2^*,Y^*$ for the dual basis of $X_1,X_2,Y$.
Obviously, $\ell(x_1)=1$ and $\ell(x_2)=0$. By Kantorovich duality,
\begin{equation}
    W_1^{d_{\rm CC}}(\mu_{x_1},\mu_e)
    \ge 
    \int \ell(x) d \mu_{x_1}(x) - \int \ell(x) d \mu_{e}(x)
    = \int (\ell(x_1 x) - \ell(x)) \, d\mu(x)
    =1\,.
\end{equation}
Suppose, by contradiction, that the inequality is an equality.
Then, for every $(x,x')$ in the support of an optimal plan $\pi$, we must have
\begin{equation}
d_{\rm CC}(x,x')=\ell(x)-\ell(x')\,.
\end{equation}
This implies that equality is realized in the direction detected by $\ell$, and hence
\begin{equation}
b(x)=b(x'),
\qquad
c(x)=c(x')\,.
\end{equation}
Therefore, $\mu_{x_1}$ and $\mu_e$ have the same $(b,c)$-marginals.
However, this is not consistent with the group structure of $H_3(\RR)$.
Indeed,
\begin{equation}
\int b(x)c(x)\,d\mu_{x_1}(x)
=
\int b(x_1x)c(x_1x)\,d\mu_e(x)
=
\int b(x)\big(b(x)+c(x)\big)\,d\mu_e(x)\,.
\end{equation}
Under our assumptions on $\mu$, we have $\int b(x)^2\,d\mu_e(x)>0$, and hence
\begin{equation}
\int b(x)c(x)\,d\mu_{x_1}(x)
\neq
\int b(x)c(x)\,d\mu_e(x)\,,
\end{equation}
which gives a contradiction.

\medskip

A slightly different approach consists in designing a nonlinear potential by adding to $a(x)$ a quadratic correction that is able to capture the non-abelian structure of the group, in order to obtain the transport bound directly from duality.
A possible choice is
\begin{equation}\label{eq:pot_con}
\ell(x)=a(x)+\frac{\eta}{R^2}b(x)c(x)
\qquad \text{for $x\in B_R^{d_{\rm CC}}(e)$ and $R\ge 10$\,,}
\end{equation}
extended by zero outside the ball of radius $10R$.
It is easy to show that $\ell$ has a $1$-Lipschitz extension with respect to $d_{\rm CC}$, for $\eta$ small enough, by studying the action of its differential in the horizontal directions $X_1$ and $X_2$.
Here one uses that $|c(x)|\le CR^2$ for every $x\in B_R^{d_{\rm CC}}(e)$.

Then, plugging $\ell$ into the duality formulation, one gains a positive term $\frac{\eta}{R^2}\int b(x)^2\, d\mu_e(x)$, exactly as in the previous calculation.
One point of care is that we need to cut off outside the ball $B_R(e)$ and use Gaussian estimates, for $R>1$ large enough, to argue that the integral outside $B_R(e)$ is negligible. Compare with Section \ref{sect:mk}.

\subsection{Comparison with Section \ref{sec:virtually_nilpotent}}

This second approach is closer to the actual proof of Theorem \ref{vecdscs:res}.
Conceptually, there are no major differences: we construct a discrete counterpart of the potential \eqref{eq:pot_con}.
Technically, however, there are two important points to keep in mind.
\begin{enumerate}
\item Although not strictly necessary, we reduce to the case of step-two groups, as in the toy model above, but now in the discrete framework.
This is a harmless reduction based on the discussion in Section \ref{subsec:quotientspaces}.
The main advantage is technical: some distance estimates become easier.

\item We introduce an auxiliary weighted distance $|\,\cdot\,|_\omega$, which is asymptotically equivalent to $|\,\cdot\,|_S$ (see Proposition \ref{vdfvsd}).
More precisely,
\begin{equation}
    \lim_{\Gamma_N\ni g\rightarrow\infty}\frac{|g|_S}{|g|_{\omega}}=1\,.
\end{equation}
This new distance is technically more convenient, since  its geodesics follow appropriately chosen directions and this allows us to  choose a $1$-Lipschitz potential with a particularly simple  expression.
At first glance, it may seem unnatural that such a noncanonical distance is asymptotically equivalent to the original one.
In fact, this is a consequence of the rigid blow-down structure of $\Gamma$ described above and the fact that $|\,\cdot\,|_\omega$ is tailored on the generating set $S$.
This is perfectly consistent with the construction in Subsection \ref{sec:malcev_closure}, which is based on the choice of a Malcev basis, although this dependence disappears in the blow-down.
\end{enumerate}

We finally mention that the blow-down structure of the Heisenberg toy model is simpler than in the general case because of its step-two structure: it coincides with the Malcev completion of $H_3(\mathbb{Z})$ with respect to a suitable Malcev basis.
In general, the Malcev completion is not stratified, whereas the blow-down is stratified.
In other words, the limiting procedure \eqref{eq_limit_lie}, which is trivial for the Heisenberg group, is non-trivial in general.

\section{The virtually nilpotent case}
\label{sec:virtually_nilpotent}

This section is devoted to the proof of Theorem \ref{vecdscs:res}.
As anticipated in the introduction, it is based on a duality argument: we construct a suitable Kantorovich potential by exploiting the nilpotent, non-abelian structure.
To this end, it is convenient, though not strictly necessary, to first reduce to the case of step-two groups.
This reduction is carried out in Section \ref{sec:reduction_step_2}.
We then introduce Malcev coordinates and an auxiliary distance.
The latter is introduced in Subsection~\ref{subsec:auxiliary_distance}. Its meaning and role have been discussed in Section \ref{sect:noemandb} in terms of the asymptotic structure of virtually nilpotent groups.

\subsection{Reduction to step two}
\label{sec:reduction_step_2}
In the notation of Section \ref{subsec:virtually}, we have the isolated lower central series \eqref{isolcent}, consisting of normal subgroups of $\Gamma$.
We apply Lemma \ref{dfcsc} with $K\defeq \Gamma_{N,3}\unlhd \Gamma$.

It turns out that $\Gamma_N/\Gamma_{N,3}$ is nilpotent, torsion-free, non-abelian, and of step two.
Indeed, it is nilpotent because it is a quotient of the nilpotent group $\Gamma_N$.
It is torsion-free by the definition of the isolated lower central series: if $g^k\in\Gamma_{N,3}$ for some $k\ge1$, then $g\in\Gamma_{N,3}$.
Moreover, relation \eqref{cdscscd} gives $[\Gamma_N,\Gamma_{N,2}]\subseteq\Gamma_{N,3}$, and therefore the commutator subgroup of $\Gamma_N/\Gamma_{N,3}$ is central.
Thus $\Gamma_N/\Gamma_{N,3}$ has nilpotency step at most two.
Finally, since $\Gamma$ is not virtually abelian, the finite-index subgroup $\Gamma_N$ is not abelian.
Hence $\Gamma_{N,2}\ne\{e\}$.
Moreover, $\Gamma_{N,2}\ne\Gamma_{N,3}$, since the isolated lower central series is strictly decreasing until it terminates; see the proof of Lemma \ref{dscc} for a clarification of this point.

Set $\overline{\Gamma}\defeq \sfrac{\Gamma}{\Gamma_{N,3}}$ and $\overline{\Gamma}_N\defeq\sfrac{\Gamma_N}{\Gamma_{N,3}}$.
Then $\overline{\Gamma}_N\unlhd\overline{\Gamma}$ and $\big|\overline{\Gamma}:\overline{\Gamma}_N\big|=|\Gamma:\Gamma_N|<\infty$.
Let $\overline S\defeq p(S)$ be the induced set of generators of $\overline{\Gamma}$.
Consider also the probability measure $\overline\mu\defeq p_*\mu$ on $\overline{\Gamma}$.

Assume that Theorem \ref{vecdscs:res} holds for the pair $(\overline{\Gamma},\overline\mu)$, which satisfies all the assumptions and has the additional property that $\overline{\Gamma}_{N,3}=\{e\}$.
Then there exists $\overline g_n\in\overline{\Gamma}$ such that
\begin{equation}
\lim_{n\to\infty}\frac{|\overline g_n|_{\overline S}}{n}\in(0,\infty)\,,
\end{equation}
and, for some $\delta>0$,   
\begin{equation}
W_1^{d_{\overline S}}(\overline\mu^{\ast n^2}_{\overline g_n},\overline\mu^{\ast n^2}_{\bar e})
\ge (1+\delta)|\overline g_n|_{\overline S}
\qquad \text{for $n\ge 0$ sufficiently large}\,.
\end{equation}
By \eqref{eq:orbit_quotient}, for every $n$ there exists $g_n\in\Gamma$ such that $p(g_n)=\overline g_n$ and
\begin{equation}
|g_n|_S=|\overline g_n|_{\overline S}\,.
\end{equation}
Therefore, by Lemma \ref{dfcsc},
\begin{equation}
W_1^{d_S}(\mu^{\ast n^2}_{g_n},\mu^{\ast n^2}_e)
\ge
W_1^{d_{\overline S}}(\overline\mu^{\ast n^2}_{\overline g_n},\overline\mu^{\ast n^2}_{\bar e})
\ge
(1+\delta)|g_n|_S\,
\end{equation}
 so that Theorem \ref{vecdscs:res} would follow for the pair $(\Gamma,\mu)$.

\subsection{Malcev coordinates}
\label{subsec:Malcev}
By Section \ref{sec:reduction_step_2}, we can focus on groups of step two.
More precisely, we consider $\Gamma$ as in Section \ref{subsec:virtually}, with the additional assumption that $\Gamma_{N,3}=\{e\}$.
In particular, $\Gamma_{N,2}$ is contained in the center of $\Gamma_N$ and is identified with $A_2\cong\mathbb{Z}^{n_2}$.

We construct a Malcev basis.
The construction is standard, the main goal of this subsection is to fix the notation that will be used throughout the proof of Theorem~\ref{vecdscs}.
We choose elements 
\begin{equation}
  x_1,\ldots,x_{n_1}\in \Gamma_N  
\end{equation}
such that their $\pi_1$-projections in $A_1$ form a basis of the corresponding $\mathbb{Z}$-module.
We then choose elements 
\begin{equation}
    y_1,\ldots,y_{n_2}\in \Gamma_{N,2}=A_2
\end{equation}
forming a basis of the corresponding $\mathbb{Z}$-module.

\begin{prop}
For every $g\in \Gamma_N$, there exist coordinates $a=a(g)\in \mathbb{Z}^{n_1}$ and $b=b(g)\in \mathbb{Z}^{n_2}$ such that
\begin{equation}
\label{eq:malcev_form}
g = x_1^{a_1} \cdots x_{n_1}^{a_{n_1}} \cdot y_1^{b_1} \cdots y_{n_2}^{b_{n_2}}\,.
\end{equation}
Moreover, this representation is unique.
\end{prop}

\begin{proof}
Define $a(g)$ by requiring that $\pi_1(g)=\sum_i a_i(g)\pi_1(x_i)$.
Then $\big(x_1^{a_1}\cdots x_{n_1}^{a_{n_1}}\big)^{-1}\cdot g\in \Gamma_{N,2}$.
Therefore, there exist unique coefficients $b(g)\in\mathbb{Z}^{n_2}$ such that
\begin{equation}
\big(x_1^{a_1}\cdots x_{n_1}^{a_{n_1}}\big)^{-1}\cdot g
= y_1^{b_1}\cdots y_{n_2}^{b_{n_2}}\,.
\end{equation}
Uniqueness follows similarly.
\end{proof}

Let $S\subseteq \Gamma$ be a finite symmetric set of generators.
The next lemma, which is standard, relates the word distance to the Malcev coordinates.

\begin{lemma}\label{lemma:iteration_growth}
There exists $C\ge 1$ such that
\begin{enumerate}
    \item $|y_i^b|_S \le C \sqrt{|b|}$ for every $b\in \mathbb{Z}$ and $1\le i\le n_2$.

    \item $|b_i(g)| \le C |g|_S^2$ for every $g\in \Gamma_N$ and $1\le i\le n_2$.

    \item $|x_i^{a}|_S \le C|a|$ for every $a\in \mathbb{Z}$ and $1\le i\le n_1$.
\end{enumerate}
\end{lemma}

\begin{proof}
By the definition of $\Gamma_{N,2}$, for every $1\le i\le n_2$ there exists a positive integer $k_i\in\mathbb{N}$ such that $y_i^{k_i}\in [\Gamma_N,\Gamma_N]$.
Hence $y_i^{k_i}$ can be written as a finite product of commutators of elements of $\Gamma_N$.
We may write these commutators as products of commutators of the form $[x_j,x_\ell]$, with $1\le j,\ell\le n_1$.
For each such commutator, we have
\begin{equation}
|[x_j,x_\ell]^b|_S \le C \sqrt{|b|}\big( |x_j|_S + |x_\ell|_S\big) \le C \sqrt{|b|} \,,
\end{equation}
which follows from the identity $[x_j^m,x_\ell^m]=[x_j,x_\ell]^{m^2}$.

Now write $b=qk_i+r$, with $0\le r<k_i$.
Since $\Gamma_{N,2}$ is central, $(y_i^{k_i})^q$ is a product of fixed commutators raised to the power $q$, while the remaining factor $y_i^r$ has uniformly bounded word length.
The previous estimate therefore gives
\begin{equation}
|y_i^b|_S \le C\sqrt{|q|}+C \le C\sqrt{|b|}\,.
\end{equation}
We now prove item $(2)$.
Let
\begin{equation}
T\defeq\{x_1^{\pm1},\dots,x_{n_1}^{\pm1},y_1^{\pm1},\dots,y_{n_2}^{\pm1}\}\,.
\end{equation}
By the Malcev-coordinate representation, $T$ is a finite symmetric generating set of $\Gamma_N$.
Up to increasing the constant $C$, we have
\begin{equation}
|g|_T\le C|g|_S
\qquad \text{for every }g\in\Gamma_N\,.
\end{equation}

Write $g\in \Gamma_N$ as a word of length $L=|g|_T$ in the alphabet $T$.
Reorder this word into Malcev normal form by moving all horizontal generators $x_j^{\pm1}$ to the left and all central generators $y_i^{\pm1}$ to the right.
Since $\Gamma_N$ has step two, every commutator produced during this reordering lies in the central subgroup $\Gamma_{N,2}$.
There are at most $C L^2$ such commutations, and each of them contributes a uniformly bounded amount to each central coordinate $b_i(g)$.
The central letters already present in the word contribute at most $L$.
Therefore
\begin{equation}
|b_i(g)|\le C L^2 \le C |g|_S^2
\qquad \text{for every }1\le i\le n_2\,.
\end{equation}

Item (3) is trivial in the case $S=T$, and, in general, follows from  the bi-Lipschitz equivalence of word-length functions with respect to different sets of generators.
\end{proof}

\subsection{Auxiliary distance}
\label{subsec:auxiliary_distance}
We keep working in the step-two framework of Subsection \ref{subsec:Malcev}: $\Gamma$ and $S$ are as in Subsection \ref{subsec:virtually}, with the additional assumption that $\Gamma_{N,3}=\{e\}$.

We adopt the notation of Section \ref{sec:poly}.
Let $E$ be the polytope in $A_1\otimes_\ZZ\RR$ with extreme set given by \eqref{vefdsc}.
Let $v_0$ be the special vertex selected in Subsection \ref{subsec:special_vertex}.
Up to relabeling $t_0, \ldots, t_m\in \Gamma$, we can assume
\begin{equation}
  \frac{\pi_1(t_0)}{l_0}=v_0\, . 
\end{equation}
We also fix a Malcev basis as in Subsection \ref{subsec:Malcev} and
we assume that the special vertex $v_0$ is aligned with the last Malcev coordinate in the first layer, namely $x_{n_1}$.
More precisely, $v_0\in \langle\pi_1(x_{n_1})\rangle$.


Let $u_i$, for $i=0,\dots,m$, be the horizontal projection of $t_i$, that is, writing $t_i$ in Malcev coordinates,
\begin{equation}\label{eq:t_i u_i}
t_i=
\underbrace{x_1^{a_{1}(t_i)}\cdots x_{n_1}^{a_{n_1}(t_i)}}_{\defeq u_i}
\cdot \, \underbrace{y_1^{b_{1}(t_i)}\cdots y_{n_2}^{b_{n_2}(t_i)}}_{\defeq z_i}\,.
\end{equation}
Since $z_i$ is central, by \eqref{vfsdscsx} and Lemma \ref{lemma:iteration_growth} we have
\begin{equation}\label{vfsdscsx1}
     |u_i^k|_S\le |t_i^k|_S+|z_i^k|_S\le  |k| l_i + C\sqrt{|k|}\qquad\text{for every $k\in\ZZ$ and $i= 0,\dots, m$}\,.
\end{equation}

Now we define the symmetric set 
\begin{equation}
    S_\omega\defeq \{u_0^{\pm 1},\dots, u_m^{\pm 1}\}\cup \{x_1^{\pm 1},\dots, x_{n_1}^{\pm 1}, y_1^{\pm 1}, \ldots , y_{n_2}^{\pm 1}\}\,.
\end{equation}
Notice that $u_i^{-1}\in S_\omega$ may fail to be horizontal, with the exception of $i=0$, where $u_0^{\pm 1}$ are both horizontal, being in the direction of the last element of the Malcev basis. Define also the symmetric weight function $\omega:S_\omega\rightarrow\RR$ as
\begin{alignat}{2}
        \omega(u_i^{\pm 1})&\defeq l_i\qquad&&\text{for $i=0,\dots,m$}\,,
        \\
        \omega(x_j^{\pm 1})&\defeq |x_j|_S\qquad&&\text{for $j=1,\dots,n_1$}\,,
        \\
        \omega(y_j^{\pm 1})&\defeq |y_j|_S\qquad&&\text{for $j=1,\dots,n_2$}\,.
\end{alignat}

  We define a weighted word-length norm as follows
\begin{equation}\label{vrefds}
    |g|_{\omega}\defeq\min\Big\{ \sum_i \omega(s_i) : g=\prod_i s_i\,, s_i\in S_\omega\}
    \qquad \text{for }g\in \Gamma_N\,,
\end{equation}
which induces naturally a left-invariant distance 
\begin{equation}
    d_{\omega}(g,h)\defeq |g^{-1}h|_{\omega}
    \qquad \text{for }g,h\in \Gamma_N\,.
\end{equation}
Clearly, there exists a constant $C\ge 1$ such that
\begin{equation}\label{eq:mild_comp}
C^{-1}|g|_{S}\le |g|_{\omega}\le C|g|_{S}\qquad\text{for every $g\in\Gamma_N$}\,.
\end{equation}
However, this estimate is very crude. The goal of this subsection is to prove a refined version at large scales, see Proposition \ref{vdfvsd} below.

\begin{lemma}
    For every $i=0,\ldots, m$, we have
    \begin{equation}\label{distu0}
        |u_i^n|_\omega = l_i |n|
        \quad
        \text{for every $n\in \ZZ$}\,.
    \end{equation}
\end{lemma}

\begin{proof}
The inequality  $(\le)$ is by definition of $|\,\cdot\,|_\omega$. We prove now the $(\ge)$ inequality, where we can of course assume that $n\ne 0$.
Write $u_i^n=s_1\cdots s_p$ such that $s_j\in S_\omega$ and $|u_i^n|_\omega=\omega(s_1)+\cdots + \omega(s_p)$. 
We then have
\begin{equation}
|n|l_i = |\pi_1(u_i^n)|_E
\le \sum_{j=1}^p |\pi_1(s_j)|_E
\le \sum_{j=1}^p\omega(s_j)
=|u_i^n|_\omega\,,
\end{equation}
where we used also \eqref{eq:E-S}.
\end{proof}

\begin{lemma}\label{lift}
There exists $C>0$ such that the following holds.
For every $v\in A_1$, there exists $\widehat v\in\Gamma_N$ such that
\begin{equation}
\pi_1(\widehat v)=v\qquad\text{and}\qquad |\widehat v|_\omega\le |v|_E+C\,.
\end{equation}
The same statement holds for $|\,\cdot\,|_S$ in place of $|\,\cdot\,|_\omega$.
\end{lemma}

\begin{proof}
If $v=0$, it is enough to take $\widehat v=e$.
We therefore assume that $v\ne0$.
By \eqref{vefdsc} and the definition of $u_i$, we can write
\begin{equation}
v=
\sum_{i=0}^m \frac{\lambda_i^+ |v|_E}{l_i}\pi_1(u_i)
+
\sum_{i=0}^m \frac{\lambda_i^- |v|_E}{l_i}\pi_1(u_i^{-1})\,,
\end{equation}
where the coefficients $\lambda_i^\pm$ are non-negative and satisfy $\sum_i(\lambda_i^++\lambda_i^-)=1$.
We define the integer coefficients
\begin{equation}
c_i^\pm\defeq \left\lfloor \frac{\lambda_i^\pm |v|_E}{l_i}\right\rfloor
\qquad\text{for } i=0,\dots,m\,.
\end{equation}
Then
\begin{equation}
\left|
v-\sum_{i=0}^m c_i^+\pi_1(u_i)-\sum_{i=0}^m c_i^-\pi_1(u_i^{-1})
\right|_E
\le 2(m+1)\,.
\end{equation}
Since the vector inside the norm belongs to the lattice $A_1$, there exist integers $a_1,\dots,a_{n_1}$ such that
\begin{equation}
v-\sum_{i=0}^m c_i^+\pi_1(u_i)-\sum_{i=0}^m c_i^-\pi_1(u_i^{-1})
=
\sum_{j=1}^{n_1} a_j\pi_1(x_j)\,.
\end{equation}
By equivalence of norms on $A_1\otimes_\ZZ\RR$ and since $\pi_1(x_1),\dots,\pi_1(x_{n_1})$ form a basis of $A_1$, we have $|a_j|\le C$ for every $j$, with $C$ independent of $v$.

We define
\begin{equation}
\widehat v
\defeq
x_1^{a_1}\cdots x_{n_1}^{a_{n_1}}
u_0^{c_0^+}\cdots u_m^{c_m^+}
u_0^{-c_0^-}\cdots u_m^{-c_m^-}\,.
\end{equation}
By construction, $\pi_1(\widehat v)=v$.
Moreover,
\begin{equation}
|\widehat v|_\omega
\le C+\sum_{i=0}^m l_i(c_i^++c_i^-)
\le C+|v|_E\,.
\end{equation}

The proof for $|\,\cdot\,|_S$ is analogous.
One uses $t_i$ in place of $u_i$ in the definition of $\widehat v$.
Since $\pi_1(t_i)=\pi_1(u_i)$, the equality $\pi_1(\widehat v)=v$ is unchanged.
The estimate follows from \eqref{vfsdscsx}, namely
\begin{equation}
|t_i^k|_S\le |k|l_i+2C_\sigma
\qquad \text{for every }k\in\ZZ\,,
\end{equation}
and from the fact that the number of factors is finite.
\end{proof}

\begin{prop}\label{vdfvsd}
    It holds that 
    \begin{align}
        |g|_{\omega} \le |g|_S +C |g|_S^{3/4}\quad\text{and}\quad|g|_{S} \le |g|_\omega +C |g|_\omega^{3/4}\qquad\text{for every }g\in \Gamma_N\,.
    \end{align}
\end{prop}
\begin{proof}
We start by proving the first inequality.
    Let $g\in\Gamma_N$ with $|g|_S=k$, so that $g=s_1\cdots s_k$, with $s_i\in S$. Then we can group the generators into $g=w_1\cdots w_{q}$, with $q\le C\sqrt{k}$, where $w_1,\dots,w_q\in\Gamma$ are such that 
    $|w_i|_S\le C\sqrt{k}$ and $\sum_{i=1}^q |w_i|_S= k$.
    Notice that $w_i\notin \Gamma_N$, in general. However, we set  
    \begin{equation}
    \begin{split}
        w_1' &\defeq w_1 \sigma(\pi(w_1))^{-1}
        \\
         w_i'&\defeq \sigma(\pi(w_1\cdots w_{i-1})) w_i\sigma(\pi(w_1\cdots w_{i-1}w_i))^{-1} \quad i=2, \ldots, q-1
         \\
         w_q '& \defeq \sigma(\pi(w_1\ldots w_{q-1})) w_q
    \end{split}
    \end{equation}
    so that $w_i'\in \Gamma_N$ and  $g'\defeq w_1'\cdots w_q'=g$. By
 \eqref{csdcdcsc},  we have
    \begin{equation}
        |w_i'|_S\le |w_i|_S +2C_\sigma\le C\sqrt{k}\,.
    \end{equation}
    We apply Lemma \ref{lift} to $\pi_1(w_i')$ to obtain $w_i''$ with $\pi_1(w_i'')=\pi_{1}(w_i')$ and 
    \begin{equation}
        |w_i''|_\omega\le |\pi_1(w_i')|_E+C\le |w_i'|_S+C\le C\sqrt{k}\,
    \end{equation} where the next-to-last inequality is due to \eqref{eq:E-S}. 
    We  set $g''\defeq w_1''\cdots w_q''$, so that
    \begin{equation}
        |g''|_{\omega}\le \sum_{i=1}^q |w_i''|_\omega\le \sum_{i=1}^q(|w_i|_S+C) \le k+C\sqrt{k}\,.
    \end{equation}
    By construction, $a_j(g'')=a_j(g')$, whereas  item 2 of Lemma \ref{lemma:iteration_growth} implies
    \begin{equation}
        \begin{split}
        |b_j(g'') - b_j(g')|
        &\le \sum_{i=1}^q |b_j(w_i'') - b_j(w_i')|
        \le \sum_{i=1}^q |b_j(w_i'')| + |b_j(w_i')|
       \\& \le C\sum_{i=1}^q |w_i''|_S^2 + |w_i'|_S^2
        \le C k^{3/2}\,.\label{eq:keyb_squareroot}
        \end{split}
    \end{equation}
    Hence, by item (1) of Lemma \ref{lemma:iteration_growth}, we deduce that
     \begin{equation}
        d_S(g,g'')=d_S(g',g'')\le C n_2\sqrt{Ck^{3/2}}\le C k^{3/4}\,.
    \end{equation}
    All in all,
    \begin{equation}
        |g|_{\omega}\le  |g''|_{\omega}+ d_{\omega}(g,g'')\le k+C\sqrt{k}+Ck^{3/4}\le |g|_S+C |g|_S^{3/4}\, .
    \end{equation}

 We now show the second inequality. 
    Now take $g\in\Gamma_N$  with $|g|_{\omega}=k$ (notice that now $k$ is not an integer, but it anyhow ranges in a discrete set).
    As before, we write
    $g=w_1\cdots w_{q}$ with $q\le C\sqrt{k}$, where $w_1,\dots, w_q\in\Gamma_N$ are such that $|w_i|_{\omega}\le C\sqrt{k}$ 
      and $\sum_{i=1}^q |w_i|_{\omega}=k$.  To keep a notational parallel with the proof above, we set $w_i'\defeq w_i$ (here, indeed, the $w_i$ are already in $\Gamma_N$). For every $i=1,\dots,q$, each $w_i'$ can be written as a word of letters in $S_\omega$ attaining the minimum in \eqref{vrefds}. If $q_\alpha\in\ZZ$ is the number of occurrences of $u_\alpha$ minus the number of occurrences of $u_\alpha^{-1}$ in such word and similarly for $r_\beta\in\ZZ$, we write
      \begin{equation}
          w_i''\defeq \prod_{\alpha=0}^m u_\alpha^{q_\alpha}\prod_{\beta=1}^{n_1} x_\beta^{r_\beta}\,.
      \end{equation}
        Notice that $\pi_1(w_i'')=\pi_1(w_i')$. We compute, using \eqref{vfsdscsx1}
      \begin{align}
          |w_i''|_S&\le \sum_{\alpha=0}^m |u_\alpha^{q_\alpha}|_S+\sum_{\beta=1}^{n_1}|x_\beta^{r_\beta}|_S\le \sum_{\alpha=0}^m (|q_\alpha|\omega(u_\alpha)+C\sqrt{|q_\alpha|})+\sum_{\beta=1}^{n_1}|r_\beta|\omega(x_\beta)\\
          &\le |w_i'|_\omega + C\sum_{\alpha=0}^m \sqrt{|q_\alpha|}\le |w_i'|_\omega+ Ck^{1/4}\,, 
      \end{align}
      where we used that $\sum_{\alpha=0}^m \sqrt{|q_\alpha|}\le C\sqrt{\sum_{\alpha=0}^m |q_\alpha|}\le C\sqrt{|w_i'|_\omega}$.  Now we set $g''\defeq w_1''\cdots w_q''$, and the proof follows by arguing as above. The key point is the following version of \eqref{eq:keyb_squareroot}:
      \begin{equation}
      |b_j(g'') - b_j(g')|
      \le C \sum_{i=1}^q\big(|w_i''|_S^2 + |w_i'|_S^2\big) \le C \sum_{i=1}^q\big(|w_i''|_\omega^2 + |w_i'|_\omega^2\big)\,,
      \end{equation}
      where, in the second inequality, we have increased the constant $C$ and used \eqref{eq:mild_comp}.
    \end{proof}

\subsection{Kantorovich potential}
\label{subsec:Kantorovich}
We work in the step-two framework of Subsections \ref{subsec:Malcev} and \ref{subsec:auxiliary_distance}. 
We define a Kantorovich potential in order to obtain the desired transport bound by duality.

The first ingredient is a linear functional 
\begin{equation}
    \ell:A_1\otimes_\ZZ\RR\to\RR
\end{equation}
satisfying
\begin{equation}\label{eq:prop_ell}
\ell(v_0)=1,\qquad
|\ell(v)|<1-\eps_0 \qquad\text{for every $v\in \Ext(E)\setminus\{\pm v_0\}$}\,,
\end{equation}
for some $0<\eps_0<1$. It exists as $\Ext(E)$ is finite; see Lemma \ref{cdescdsacsc}.
Notice that
\begin{equation}\label{eq:bound_ell}
|\ell \circ \pi_1(g)| \le |\pi_1(g)|_E \le |g|_S
\qquad\text{for every } g\in \Gamma_N\,,
\end{equation}
where in the last inequality we used \eqref{eq:E-S}.

Thanks to Lemma \ref{dscc}, recalling \eqref{dscsc1}, we can choose a linear operator 
\begin{equation}\label{eq:beta}
  \beta:A_2\otimes_\ZZ\RR\to\RR  
\end{equation}
such that the linear operator
\begin{equation}
\alpha:A_1\otimes_\ZZ \RR \to \RR,
\qquad
\alpha(\,\cdot \,)\defeq\beta([v_0\, ,\cdot \,]_{\Gamma_N})
\end{equation}
is non zero.
Recall that, in our step-two framework, $A_2$ is identified with $\Gamma_{N,2}$.
Hence we can define $\beta(g)$ for every $g\in\Gamma_N$ by means of the Malcev coordinates 
\begin{equation}
\beta(g)\defeq \beta\big(y_1^{b_1(g)}\cdots y_{n_2}^{b_{n_2}(g)}\big)
\qquad\text{for } g\in\Gamma_N\,.
\end{equation}
Notice that $\beta(g)$ is linear as a function of the second-layer Malcev coordinate $b$, but it is not a group homomorphism unless $\Gamma_N$ is abelian.

We use the same convention as in Subsection \ref{subsec:auxiliary_distance}, namely $v_0\in \langle\pi_1(x_{n_1})\rangle$.

\begin{defn}[Potential]
\label{def:Psi}
Fix a scale $R>0$ and $0<\eta<1$. We define $\Psi=\Psi_{\eta,R}:\Gamma_N \to \RR$ as
\begin{equation}   
\Psi(g)\defeq \ell(\pi_1(g))+ \frac{\eta}{R^2} \alpha(\pi_1(g))\, \beta(g)\, .
\end{equation} 
\end{defn}

Notice that, in Malcev coordinates, $\ell\circ\pi_1$ and $\alpha\circ\pi_1$ depend only on $a\in\mathbb{Z}^{n_1}$, whereas $\beta$ depends only on $b\in\mathbb{Z}^{n_2}$.
We finally remark that $\ell$ is independent of the fixed Malcev basis, although it depends on the non-unique choice of $v_0$, whereas $\alpha$ and $\beta$ depend on the fixed Malcev basis.
This is not an issue, since this dependence disappears at large scales.

\begin{lemma}\label{LipLem}
For $\eta\in(0,1)$ small enough, depending only on $\Gamma$ and $S$, the potential $\Psi=\Psi_{\eta,R}$ is $1$-Lipschitz with respect to $d_\omega$ on $B_R^{d_\omega}(e)$, for every $R\ge 1$.
\end{lemma}

\begin{proof}
    It is enough to show that if $g\in B_R(e)$ and $s\in S_\omega$, then
    \begin{equation}
        \big|\Psi(gs)-\Psi(g)\big|\le \omega(s)\,.
    \end{equation}
    We can compute,  exploiting linearity
    \begin{align}
       \Psi(gs)-\Psi(g)&=  \ell (\pi_1(gs))- \ell (\pi_1(g))+
       \frac{\eta}{R^2}(\alpha (\pi_1(gs))-\alpha (\pi_1 (g)))\beta(gs)
       \\& \qquad +\frac{\eta}{R^2}\alpha( \pi_1(g))(\beta(gs)-\beta(g))\\
       &= \ell(\pi_1(s))+\frac{\eta}{R^2}\alpha( \pi_1(s))\beta(gs)+\frac{\eta}{R^2}\alpha(\pi_1(g))(\beta(gs)-\beta(g))\,.\label{eq:keypsi}
    \end{align}
We will use repeatedly that, for every $g\in B_R(e)$ and $s\in S_\omega$, it holds
    \begin{equation}
        |\alpha(\pi_1(s))| \le C\,,
        \quad
        |\alpha(\pi_1(g))|\le C R\, ,
        \quad |\beta(gs)|\le CR^2\,,
        \quad
        |\beta(gs)-\beta(g)|\le CR\,.
    \end{equation}
The first two inequalities follow from the linearity of $\alpha\circ \pi_1 : A_1\otimes_\ZZ \RR \to \RR$, \eqref{eq:E-S} and \eqref{eq:mild_comp}:
\begin{equation}
    |\alpha(\pi_1(g))| = |\beta([v_0, \pi_1(g)])|
    \le C|\pi_1(g)|_E
    \le C| g |_S
    \le C |g |_\omega\,.
\end{equation}
The third inequality follows from the linearity of $\beta(g)$ with respect to the second-layer Malcev coordinate $b(g)$ and Lemma \ref{lemma:iteration_growth}:
\begin{equation}
    |\beta(g)| \le C|b(g)| \le C |g|_S^2 \le C |g|_\omega^2\,.
\end{equation}
We now prove the fourth inequality.
Since $s\in S_\omega$, its Malcev coordinates are uniformly bounded.
Moreover, in a step-two group, the second-layer coordinate of a product differs from the sum of the second-layer coordinates only by the commutator contribution of the first-layer coordinates.
Therefore, using item 3 in Lemma \ref{lemma:iteration_growth}, we obtain
\begin{equation}
|\beta(gs)-\beta(g)|
\le C\big(1+|a(g)|\big)
\le C\big(1+|g|_\omega\big)\,.
\end{equation}

To estimate \eqref{eq:keypsi}, we distinguish the various cases. 
    \begin{enumerate}[label=\roman*)]
        \item $s= u_0^{\pm 1}$. Hence, $\pi_1(s)=\pm l_0v_0$ and $\ell(s)=\pm l_0$. Moreover, $\alpha(s)=0$. Finally, as $u_0$ is in the direction of the last element of the Malcev basis, it follows that $\beta(gs)-\beta(g)=0$. All in all,
        $|\Psi(gs)-\Psi(g)|=l_0=\omega(s)$.
        
        \item $s= u_i^{\pm 1}$, $i\ne 0$. By construction,  $|\ell(\pi_1(u_i^{\pm 1}))|\le( 1-\eps_0)l_i$, so we see that
        \begin{equation}
    \big|\Psi(gs)-\Psi(g)\big|\le (1-\eps_0)\omega(s)+\frac{\eta}{R^2}  C R^2+\frac{\eta}{R^2}(CR)^2\le \omega(s)\,,
        \end{equation}
        provided that $\eta\in (0,1)$ is small enough.

    \item $s= x_j^{\pm 1}$, $1\le j\le n_1-1$. By construction,  there exists $\eps>0$ such that 
    \begin{equation}
      |\ell(\pi_1(s))|\le (1-\epsilon)|\pi_1(s)|_E
    \le (1-\eps)|s|_S=(1-\epsilon)\omega(s)  
    \end{equation}
    so that the conclusion is as in $\rm {ii)}$.
    \item $s= x_{n_1}^{\pm 1}$. Then $|\ell(s)|=|\pi_1(s)|_E\le |s|_S=\omega(s)$, and the conclusion is as in $\rm {i)}$.
    
    \item $s= y_j^{\pm 1}$, for $1\le j\le n_2$. We notice that $\ell(s)=\alpha(s)=0$, so that 
    \begin{equation}
        |\Psi(gs)-\Psi(g)|
        \le \frac{\eta}{R^2}(CR)^2\le \omega(s)\,,
    \end{equation}
    provided that $\eta\in (0,1)$ is small enough.\qedhere
    \end{enumerate}
\end{proof}

\subsection{The main computation} \label{subsec:mainprop}
We work in the step-two framework of  Subsections \ref{subsec:Malcev}, \ref{subsec:auxiliary_distance} and \ref{subsec:Kantorovich}, with the same notation. 
Fix $\eta\in(0,1)$ small enough, given by Lemma \ref{LipLem}. Fix $H\in\NN$ positive to be chosen later. We define
\begin{equation}
    \widetilde{\Psi}_{n,H}\defeq
    \begin{cases}
        \Psi_{\eta, {2Hn}}\qquad&\text{on }B^{d_\omega}_{2Hn}(e)\,,\\
        0\qquad&\text{on }\Gamma_N\setminus B^{d_\omega}_{4Hn}(e)\,,
    \end{cases}
\end{equation}
extended to be $1$-Lipschitz with respect to $d_\omega$ on $\Gamma_N$. Notice in particular that 
\begin{equation}\label{eq:tildePsi_bound}
  |\widetilde \Psi_{n,H}(g)|\le 4Hn
  \qquad\text{for every } g\in \Gamma_N\,.
\end{equation}

\begin{prop}\label{KeyLemma}
Under the assumptions above, provided that $H\ge C$,
\begin{equation}
        \liminf_{n\rightarrow\infty}\frac{1}{l_0n}\sum_{g\in\Gamma_N}\big(\widetilde\Psi_{n,H}(u_0^n g)-\widetilde\Psi_{n,H}( g)\big)\nu^{(n^2)}(g)>1\,.
    \end{equation}
\end{prop}
\begin{proof}
    We abbreviate $B_{R}$ as $B_R^{d_\omega}(e)$. We compute
    \begin{equation}\label{movfdesc}
    \begin{split}
                &\frac{1}{l_0n}\sum_{g\in\Gamma_N}\big(\widetilde\Psi_{n,H}(u_0^n g)-\widetilde\Psi_{n,H}( g)\big)\nu^{(n^2)}(g)\\
        &\qquad\qquad\ge  \frac{1}{l_0n} \sum_{g\in B_{Hn}}\big(\widetilde\Psi_{n,H}(u_0^n g)-\widetilde\Psi_{n,H}( g)\big)\nu^{(n^2)}(g)-{8H}l_0^{-1}\nu^{(n^2)}(\Gamma_N\setminus B_{Hn})\\
        &\qquad\qquad\ge \frac{1}{l_0n}\sum_{g\in B_{Hn}}\big(\Psi_{\eta, 2Hn}(u_0^n g)-\Psi_{\eta, 2Hn}( g)\big)\nu^{(n^2)}(g)-{CH}e^{-H^2/C}\,,
            \end{split}
        \end{equation}
        where we used \eqref{eq:tildePsi_bound}, Lemma \ref{vfedsc} and assumed $H\ge l_0$.
        Since $\pi_1(u_0^n)$ is aligned with $v_0$, we have
        \begin{equation}
            \alpha(u_0^ng)=\beta([v_0, \pi_1(u_0^ng)]_{\Gamma_N}) = \alpha(g)\,,
        \end{equation}
        hence, recalling Definition \ref{def:Psi} and properties \eqref{eq:prop_ell}, we obtain
         \begin{align}
            \Psi_{\eta, 2Hn}(u_0^n g)-\Psi_{\eta, 2Hn}( g)
            &=\ell(u_0^n)
            +\frac{\eta}{(2Hn)^2}\alpha(g)(\beta(u_0^ng)-\beta(g))
            \\&= l_0 n + \frac{\eta}{(2Hn)^2}\alpha(g)(\beta(u_0^ng)-\beta(g))\,.
        \end{align}
        Now, notice that 
        \begin{equation}
           u_0^n g=[u_0^n,g]gu_0^n=gu_0^n[u_0^n,g]\,.
        \end{equation}
        Since $u_0$ is a power of $x_{n_1}$, we deduce
        \begin{equation}
            \beta(u_0^n g)=\beta(gu_0^n)+\beta([u_0^n,g])=\beta(g)+\beta([u_0^n,g])
            = \beta(g) + n {l_0}\beta([v_0,\pi_1(g)]_{\Gamma_N})
        \end{equation}
        leading to
        \begin{align}
            \frac{1}{l_0n}\Big(\Psi_{\eta,2Hn}(u_0^n g)-\Psi_{\eta, 2Hn}( g)\Big)
            &= 1+\frac{1}{l_0n}\frac{\eta}{(2Hn)^2}\alpha(g)\, n {l_0}\beta([v_0,\pi_1(g)]_{\Gamma_N})\\
            & 
           = 1+\frac{\eta}{(2Hn)^2}\alpha^2(g)\,.
        \end{align}
        Hence,  using  H\"older's inequality,
        \begin{align}
           & \frac{1}{l_0n}\sum_{g\in B_{Hn}}\big(\Psi_{\eta, 2Hn}(u_0^n g)-\Psi_{\eta, 2Hn}( g)\big)\nu^{(n^2)}(g)= \sum_{g\in B_{Hn}}\Big(1+\frac{\eta}{(2Hn)^2}\alpha^2(g)\Big)\nu^{(n^2)}(g)\\
           &\qquad\qquad =\sum_{g\in \Gamma_N}\Big(1+\frac{\eta}{(2Hn)^2}\alpha^2(g)\Big)\nu^{(n^2)}(g)-\sum_{g\in \Gamma_N\setminus B_{Hn}}\Big(1+\frac{\eta}{(2Hn)^2}\alpha^2(g)\Big)\nu^{(n^2)}(g)\\
           &\qquad\qquad \ge  \Big(1+\frac{\eta}{(2Hn)^2}\sum_{g\in\Gamma_N}\alpha^2(g)\nu^{(n^2)}(g)\Big)-\nu^{(n^2)}(\Gamma_N\setminus B_{Hn})\\
           &\qquad\qquad\qquad\qquad-\frac{\eta}{(2Hn)^2}\nu^{(n^2)}(\Gamma_N\setminus B_{Hn})^{1/2} \Big(\sum_{g\in \Gamma_N}\alpha^4(g)\nu^{(n^2)}(g)\Big)^{1/2}\,.
        \end{align}
Using Lemma \ref{vfedsc} and Lemma \ref{scacs}  
         twice, if $n$ is large enough (which we assume from now on), we continue the above  as
        \begin{align}
            &\frac{1}{l_0n}\sum_{g\in B_{Hn}}\big(\Psi_{\eta,2Hn}(u_0^n g)-\Psi_{\eta,2Hn}( g)\big)\nu^{(n^2)}(g)\\
            &\qquad\qquad\ge \Big(1+\frac{\eta}{(2Hn)^2}n^2 /C\Big)- Ce^{-H^2/C}-\frac{\eta}{(2Hn)^2}Ce^{-H^2/2C}Cn^2\,.
        \end{align}
        We can use this inequality to continue \eqref{movfdesc} as
        \begin{align}
             &\frac{1}{l_0n}\sum_{g\in\Gamma_N}\big(\widetilde\Psi_{n,H}(u_0^n g)-\widetilde\Psi_{n,H}( g)\big)\nu^{(n^2)}(g)\\
             &\qquad\qquad\ge \Big(1+\frac{\eta}{C(2H)^2} \Big)-Ce^{-H^2/C}-\frac{\eta}{(2H)^2}Ce^{-H^2/2C}-CHe^{-H^2/C}>1  \,,      \end{align}
       provided that $H$ is large enough.
\end{proof}

\subsection{Proof of Theorem \ref{vecdscs:res}}\label{set:prooffive}

Recall that, by Section \ref{sec:reduction_step_2}, it is enough to treat the case of step-two groups, i.e., $\Gamma_{N,3}=\{e\}$.
We consider the framework and notation of Subsections \ref{subsec:Malcev}, \ref{subsec:auxiliary_distance},  \ref{subsec:Kantorovich} and \ref{subsec:mainprop}.
Thus $\Gamma$ is finitely generated, virtually nilpotent, and not virtually abelian, with $\Gamma_N$ of step two.
Moreover, $S$ is a symmetric set of generators, and $E$ is the associated polytope with special vertex $v_0\in A_1\otimes_\ZZ\RR$.
We fix a Malcev basis with $x_{n_1}$ aligned with $v_0$, and define the elements $u_i$ from the vertices of $E$ as in Subsection \ref{subsec:auxiliary_distance}, see \eqref{eq:t_i u_i}.
Finally, $\mu$ is a symmetric probability measure on $\Gamma$ with finite support containing a set of generators  and $e$.
Let $\nu^{(n)}$ be as in \eqref{eq:nu_n}. We are going to repeatedly  use \eqref{distu0} for $i=0$, i.e., $|u_0^n|_\omega=l_0|n|$ for every $n\in\NN$.

\medskip

We first show that, under the assumptions above, we have the claim of Theorem \ref{vecdscs:res} with $|\,\cdot\,|_S$ replaced by $|\,\cdot\,|_\omega$ and $\mu^{\ast n^2}$ replaced by $\nu^{(n^2)}$. Namely, we show that
\begin{equation}\label{vefdscc}
\liminf_{n\rightarrow\infty}\frac{W_1^{d_\omega}(\nu^{(n^2)}_{u_0^n},\nu_e^{(n^2)})}{|u_0^n|_\omega}>1\,.
\end{equation}
To this aim, we recall  the easy part of duality in optimal transport. Let  $\pi$ be an optimal plan for $\nu^{(n^2)}_{u_0^n},\nu^{(n^2)}_e$ and take $f:\Gamma_N\rightarrow\RR$, a $1$-Lipschitz function with respect to $d_\omega$. Then
    \begin{align}
        W_1^{d_\omega}(\nu^{(n^2)}_{u_0^n},\nu^{(n^2)}_e)&=\int_{\Gamma_N\times\Gamma_N} d_\omega(g,h) d\pi(g,h)\\
        &\ge \int_{\Gamma_N\times\Gamma_N} (f(g)-f(h))d \pi(g,h)\\
        &= \sum_{g\in\Gamma_N} f(g)\nu^{(n^2)}_{u_0^n}(g)-\sum_{h\in\Gamma_N} f(h)\nu^{(n^2)}_e(h)\\
        &= \sum_{g\in\Gamma_N} f(u_0^ng)\nu^{(n^2)}(g)-\sum_{g\in\Gamma_N} f(g)\nu^{(n^2)}(g)\,.
    \end{align}
    
    Hence, since  $\widetilde{\Psi}_{n,H}$ is $1$-Lipschitz with respect to $d_\omega$,  Proposition \ref{KeyLemma} implies that
    \begin{equation}
       \liminf_{n\rightarrow\infty}\frac{ W_1^{d_\omega}(\nu^{(n^2)}_{u_0^n},\nu^{(n^2)}_e)}{l_0n}\ge\liminf_{n\rightarrow\infty}\frac{1}{l_0n}\sum_{g\in\Gamma_N}\big(\widetilde\Psi_{n,H}(u_0^n g)-\widetilde\Psi_{n,H}( g)\big)\nu^{(n^2)}(g)>1\,,
    \end{equation}
    provided that $H$ is large enough. We fix such $H$ and  \eqref{vefdscc} follows.
\medskip

Fix $\eps\in(0,1)$. By Proposition \ref{vdfvsd}, for every $g\in\Gamma_N$, we have
         \begin{equation}\label{ineq1}
             |g|_S\ge |g|_\omega-C|g|_S^{3/4}\ge |g|_\omega-C|g|_\omega^{3/4}\ge (1-\eps)|g|_\omega-C\eps^{-3}
         \end{equation}
and similarly
         \begin{equation}\label{ineq2}
             |g|_\omega \ge (1-\eps)|g|_S-C\eps^{-3}\,.
         \end{equation}
 By \eqref{ineq1} and \eqref{ineq2}, we have 
         \begin{equation}
             \frac{W_1^{d_S}(\nu^{(n^2)}_{u_0^n},\nu^{(n^2)}_e)}{|u_0^n|_S}\ge \frac{(1-\epsilon)W_1^{d_\omega}(\nu^{(n^2)}_{u_0^n},\nu^{(n^2)}_e)-C\epsilon^{-3}}{(1-\epsilon)^{-1}(|u_0^n|_\omega+C\epsilon^{-3})}
         \end{equation}
         so that \eqref{vefdscc} implies that
         \begin{equation}\label{vbrfdssdc}
        \liminf_{n\rightarrow\infty}\frac{W_1^{d_S}(\nu^{(n^2)}_{u_0^n},\nu^{(n^2)}_e)}{|u_0^n|_S}>1\,,
         \end{equation}
provided that $\eps\in(0,1)$ is small enough. 
By \eqref{brfdv}, 
         \begin{equation}\label{ineq0}
             W_1^{d_S}(\mu^{\ast n^2}_{u_0^n},\mu^{\ast n^2}_e)\ge W_1^{d_S}(\nu^{(n^2)}_{u_0^n},\nu^{(n^2)}_e)-2C_\sigma\qquad\text{for every $n\in\NN$}\,,
         \end{equation}
so that we have
         \begin{equation}
        \liminf_{n\rightarrow\infty}\frac{W_1^{d_S}(\mu^{\ast n^2}_{u_0^n},\mu^{\ast n^2}_e)}{|u_0^n|_S}>1\,.
         \end{equation}
         
   Now, using   \eqref{ineq1} and \eqref{ineq2}, we see that
 \begin{equation}
    \lim_{n\rightarrow\infty} \frac{|u_0^n|_S}{n}=l_0\in(0,\infty)\,,
 \end{equation}
 which concludes the proof.
\qed

\begin{remark}
 A quick inspection of the proof  above shows that  \eqref{vbrfdssdc}
holds for any sequence of probability measures $(\nu^{(n)})_n$ satisfying the conclusions of  Lemma \ref{vfedsc} and Lemma \ref{scacs},
with the convention that     $\nu^{(n)}_g=(L_g)_*\nu^{(n)}$.
\end{remark}

\section{The virtually abelian case}

\subsection{Proof of Theorem \ref{cdscs:rest}}
 Notice first that 
 \begin{equation}
     W_1^{d_S}(\mu^{\ast n}_e,\mu^{\ast (n+1)}_e)\le C\,,
 \end{equation}
 which is justified by the plan $(h,k)\mapsto \mu^{\ast n}_e(h)\mu_e(h^{-1}k)$. Hence, up to replacing $\mu$ with $\mu^{\ast p}$, we see that we can assume that  $\pi_{\GGN}(\supp(\mu))=\GGN$. In particular, there exists $\alpha\in (0,1)$ such that

 \begin{equation}\label{sdacdcds}
     \sum_{h\in\Gamma_N}\mu(h\sigma(b))>\alpha\qquad\text{for every $b\in\GGN$}\,,
 \end{equation}
 which gives a uniform bound from below for transition probabilities on $\GGN$.
     
Write first the estimate
\begin{equation}\label{vfd1}
            \big|W_1^{d_S}(\mu^{\ast n}_e,\mu^{\ast n}_g)- W_1^{d_S}(\mu^{\ast n}_e,\mu^{\ast n}_{g_N})\big|\le W_1^{d_S}(\mu^{\ast n}_{g_N},\mu^{\ast n}_g)\,.
    \end{equation}
    Notice that by \eqref{brfdv}, for every $h\in\Gamma_N$ and $n\in\NN$,
    \begin{equation}\label{vfd2}
        \big|W_1^{d_S}(\mu^{\ast n}_e,\mu^{\ast n}_h)-W_1^{d_S}(\nu^{(n)}_e,\nu^{ (n)}_h)\big|\le 2C_\sigma\,,
    \end{equation}
    Moreover, for $h\in\Gamma_N$,
    \begin{equation}
        W_1^{d_S}(\nu^{(n)}_e,\nu^{ (n)}_h)\le|h|_S\,,
    \end{equation}
    which is verified by the translation plan, since $\Gamma_N$ is abelian. Also, Lemma \ref{lift} (as $\Gamma_N$ is abelian), implies that 
    \begin{equation}
       |h|_E\le |h|_S\le |h|_E+C\qquad\text{for every }h\in \Gamma_N\,.
    \end{equation}
    Clearly, $|\,\cdot\,|_E$ induces a distance on $\Gamma_N$, $d_E$, and, for $h\in\Gamma_N$,
    \begin{equation}
                W_1^{d_E}(\nu^{(n)}_e,\nu^{ (n)}_h)\ge|h|_E\,,
    \end{equation}
    which is verified by duality, through a linear functional $\ell\in (A\otimes\RR)^*$ with $\|\ell\|_{(A\otimes\RR,|\,\cdot\,|_E)^*}=1$ and $\ell(h)=|h|_E$. The computations are  similar as we did for the nilpotent case and we omit the details as in this case the computation is much simpler.

    The three inequalities above imply that
    \begin{equation}\label{vfd3}
        \big|W_1^{d_S}(\nu^{(n)}_e,\nu^{ (n)}_h)-|h|_S\big|\le C\qquad\text{for every $n\in\NN$ and $h\in\Gamma_N$}\,.
    \end{equation}
    Hence, from \eqref{vfd1}, \eqref{vfd2} and \eqref{vfd3}, we have that, 
    \begin{equation}
        \big|W_1^{d_S}(\mu^{\ast n}_e,\mu^{\ast n}_g)-|g_N|_S\big|\le W_1^{d_S}(\mu^{\ast n}_{g_N},\mu^{\ast n}_g)+C\,.
    \end{equation}

    By the above and \eqref{csdcdcsc}, it is enough to prove that for every $b\in\GGN$, 
    \begin{equation}
        W_1^{d_S}(\mu^{\ast n}_e,\mu^{\ast n}_{\sigma(b)})\le C\qquad\text{for every $n\in\NN$}\,.
    \end{equation}
    As $\GGN$ is finite, we reduce to show that for every $s_0\in\Gamma$ fixed,
 \begin{equation}
        W_1^{d_S}(\mu^{\ast n}_e,\mu^{\ast n}_{s_0})\le C\qquad\text{for every $n\in\NN$}\,.
    \end{equation}

    It turns out that to prove the claim, it is convenient to use the probabilistic interpretation of optimal transport.  
      Let now $(\xi_i)_{i\ge 1},(\eta_i)_{i\ge 1}$ be i.i.d.\ random variables with distribution $\mu$, i.e.,
    \begin{equation}
        \mathbb P(\xi_i=s)=\mathbb P(\eta_i=s)=\mu(s)\qquad\text{for every $s\in \Gamma$ and $i\ge 1$}\,.
    \end{equation}
    Let also 
    \begin{equation}
        \tau\defeq\min\{i\in\NN: \pi(\xi_1\cdots \xi_i)=\pi(s_0\eta_1\cdots \eta_i)\}\,.
    \end{equation}
    Notice that by \eqref{sdacdcds} $\tau$ is finite a.e., more precisely 
    \begin{equation}\label{cdsc}
    \mathbb{E}        (\tau)\le
    \frac{1}{\alpha}\,,
    \end{equation}
    as, for every $i$, the probability that $\eta_i$ is so that $\pi(s_0\eta_1\cdots \eta_{i-1}\eta_i)=\pi(\xi_1\cdots\xi_i)$ is bounded from below by $\alpha$ (thanks to \eqref{sdacdcds}), so that the  stopping time $\tau$ is bounded from above by a geometric random variable of parameter $\alpha$.

    We are going to explicitly construct a coupling.      
    We define, for every $n\in\NN$,
    \begin{equation}
        X^n_e\defeq \xi_1\cdots \xi_n
    \end{equation} 
    and
    \begin{equation}
         Y^n_{s_0}\defeq\begin{cases}
             s_0\eta_1\cdots \eta_{\tau}\xi_{\tau+1}\cdots \xi_{n}\qquad&\text{if $n> \tau$}\,,\\
             s_0\eta_1\cdots \eta_{n}\qquad&\text{if $n\le \tau$}\,.
         \end{cases} 
    \end{equation}
    Notice that $X^n_{e}$ has law  $\mu^{\ast n}_e$ and $Y^n_{s_0}$ has  law $\mu^{\ast n}_{s_0}$, since the increments are i.i.d.\ with law $\mu$. Trivially, for $n\le \tau$
    \begin{equation}
        d_S(X^n_e,Y^n_{s_0})\le C+Cn\le C+C\tau\,.
    \end{equation}
    Now we deal with the case $n>\tau$.
    We then notice that by definition of $\tau$,
    \begin{equation}
        (\xi_1\cdots \xi_\tau)^{-1} s_0\eta_1\cdots \eta_\tau\in\Gamma_N\,,
    \end{equation}
    so that, as $\Gamma_N$ is abelian,  for $n>\tau$,
    \begin{align}
        d_S(X^n_e,  Y^n_{s_0})&=|(\xi_1\cdots\xi_{n})^{-1} s_0 \eta_1\cdots \eta_{\tau}\xi_{\tau+1}\cdots\xi_n|_S\\
        &=|(\xi_{\tau +1}\cdots\xi_{n})^{-1} \underbrace{(\xi_1\cdots \xi_\tau)^{-1} s_0 \eta_1\cdots \eta_{\tau}}_{\in\Gamma_N}\underbrace{\xi_{\tau+1}\cdots\xi_n\sigma(b_n)^{-1}}_{\in\Gamma_N}\sigma(b_n)|_S\\
        &=|(\xi_{\tau +1}\cdots\xi_{n})^{-1}\xi_{\tau+1}\cdots\xi_n\sigma(b_n)^{-1} (\xi_1\cdots \xi_\tau)^{-1} s_0 \eta_1\cdots \eta_{\tau}\sigma(b_n)|_S\\
        &=|\sigma(b_n)^{-1}(\xi_1\cdots \xi_\tau)^{-1}s_0 \eta_1\cdots \eta_{\tau}\sigma(b_n)|_S\le C+C\tau\,,
    \end{align}
    where we denoted  $b_n\defeq\pi(\xi_{\tau+1}\cdots \xi_n)$.

    We then use  these estimates in 
    \begin{align}
        W_1^{d_S}(\mu^{\ast n}_e,\mu^{\ast n}_{s_0} )\le \mathbb E (d_S(X_e^n, Y^n_{s_0}))\le C+C\mathbb E\tau \le C\,,
    \end{align}
    where we used also \eqref{cdsc}.\qed

\subsection{Proof of Remark \ref{rempianofurbo}}
\label{subsec:proof_remark}

It is convenient to write the infinite dihedral group $\Gamma$ from Example \ref{infdie} as a semi-direct product $\ZZ\rtimes C_2$, where $C_2\defeq \{0,1\}$ is cyclic. The only non-trivial semi-direct product is 
\begin{equation}
    (a_1,\epsilon_1)(a_2,\epsilon_2)\defeq (a_1+(-1)^{\epsilon_1}a_2,\epsilon_1+\epsilon_2)\,.
\end{equation}
We consider the finite symmetric set of generators 
\begin{equation}
    S\defeq \{(1,0),(-1,0),(0,1)\}\,
\end{equation}
and the measure $\mu$ given by the $1/2$-lazy symmetric random walk as in \eqref{asavefdscz}.
We set $g_n\defeq(n,1)$ (notice that $|g_n|_S=n+1$) and we want to estimate
\begin{equation}
    \sum_{g\in\Gamma} d_S(g,g_ng)\mu^{\ast n^2}(g)=\sum_{g\in\Gamma}  |g^{-1}g_ng|_S\mu^{\ast n^2}(g)\,,
\end{equation}
which is the value given by the trivial translation plan. 
Notice that 
\begin{equation}
    |(a,\epsilon)^{-1}(n,1)(a,\epsilon)|_S=1+|n-2a|\,,
\end{equation}
so that
\begin{equation}
    \begin{split}
        \sum_{g\in\Gamma}  |g^{-1}g_ng|_S\mu^{\ast n^2}(g)=&1+ \sum_{g\in\Gamma} |n-2a(g)| \mu^{\ast n^2}(g)= 1+\mathbb{E}(|n-2a(X^{n^2})|)\\
        &=1+\mathbb{E}(|n-2M_{n^2}|)\,,
         \end{split}
\end{equation}
where $a(g)$ is the first component of $g=(a(g),\eps(g))\in \Gamma$,
 $X^m$  is the $m$-step $1/2$-lazy symmetric random walk originating from $e$ with respect to $S$ (see \eqref{aaavefdsc}), and  $M_m\defeq a(X^m)$.  We have that $M_m$ is the $2/3$-lazy symmetric random walk on $\ZZ$ starting from $0$, namely,
\begin{equation}
    M_{m+1}=\begin{cases}
        M_m\qquad&\text{with probability $2/3$}\,,\\
        M_m+1\qquad&\text{with probability $1/6$}\,,\\
        M_m-1\qquad&\text{with probability $1/6$}\,.
    \end{cases}
\end{equation}
In particular, $M_m$ has zero mean.

By the central limit theorem, we have convergence to the Gaussian random variable:
\begin{equation}
    \frac{M_{n^2}}{n}\rightarrow Z\sim N\Big(0,\frac{1}{3}\Big)\,.
\end{equation}
Hence, for some $\delta>0$,
\begin{equation}
   \frac{\mathbb{E}(|n-2M_{n^2}|)}{n}= \mathbb{E}(|1-\sfrac{2}{n}M_{n^2}|)\rightarrow \mathbb{E}(|1-2Z|)=1+2\delta\qquad\text{as $n\rightarrow\infty$}\,.
\end{equation}
Notice that even though $t\mapsto|1-2t|$ is unbounded, the convergence of expected values is rigorously justified  thanks to the fact that $\mathbb{E}((\sfrac{1}{n}M_{n^2})^2)=1/3$.

All in all, for $n$ large enough,
\begin{equation}
    \sum_{g\in\Gamma}  |g^{-1}g_ng|_S\mu^{\ast n^2}(g)\ge 1+(1+\delta) n\,,
\end{equation}
which concludes the proof.\qed

\subsection{Proof of Theorem \ref{thm:virtabel1:rest}}
    Of course, it is enough to treat the case where  $\Gamma_N$ is infinite, otherwise $\Gamma$ would be finite and the claim trivial.
    
    Recall that we denote with $\pi:\Gamma\rightarrow\GGN$ the projection. First, we build  subsets $(U_b)_{b\in\GGN\setminus\{e\}}$ such that, for every $b\in\GGN$, $b\ne e$,
    \begin{itemize}
    \item $U_b\subseteq \pi^{-1}(b)$,
        \item $|U_b|=2$,
        \item $U_b^{-1}\defeq \{g^{-1}:g\in U_b\}=U_{b^{-1}}$.
    \end{itemize}
    Moreover, for every $b_1,b_2\in\GGN\setminus\{e\}$, $b_1\ne b_2$, we fix a bijection $\theta_{b_1,b_2}:U_{b_1}\rightarrow U_{b_2}$.
     
    To show that this is possible, notice first that for $b\in\GGN$, $\pi^{-1}(b)\cong\Gamma_N$ is infinite.  If $b\ne b^{-1}$, define $U_b$ by choosing any two elements of $\pi^{-1}(b)$, and define $U_{b^{-1}}\defeq U_b^{-1}$. If instead $b=b^{-1}$, take $g_1,g_2\in\pi^{-1}(b)$, with $g_1\ne g_2$. If $g_1\ne g_1^{-1}$, set $U_b=\{g_1, g_1^{-1}\}$; if $g_1= g_1^{-1}$ but $g_2\ne  g_2^{-1}$, set $U_b=\{g_2, g_2^{-1}\}$; if instead $g_1= g_1^{-1}$ and $g_2= g_2^{-1}$, set $U_b=\{g_1, g_2\}$.

    Let now $e\in B\subseteq\Gamma_N$ be any finite, symmetric, conjugation invariant subset containing a set of generators. This is possible as conjugacy classes are finite, since $\Gamma$ is virtually abelian.
   Define, for $D\in\NN$ to be chosen later, the finite and symmetric set
    \begin{equation}
        S\defeq \Big( B^D \cup\bigcup_{b\in\GGN\setminus \{e\}} U_b\Big)\setminus\{e\}\,.
    \end{equation}
    Notice that $B^D$, the $D$-fold product $B\cdot \ldots \cdot B$, is still finite and conjugation invariant. Moreover, $S$ is a finite symmetric set of generators.
    
        We want to show that \eqref{efrvcrdscf} holds, provided  that we fix $D$ large enough. We know that it is enough to verify \eqref{efrvcrdscf} when $n=1$ and $g=s_0\in S$. As $\supp(\mu_e)= S\cup\{e\}$, it  is enough to  construct a map 
        $\Phi:S\cup\{e\}\rightarrow\Gamma$ that satisfies 
        \begin{equation}\label{vecdscc}
           \Phi_{*}\mu_e= \mu_{s_0}\qquad\text{and}\qquad d_S(t,\Phi(t))\le 1\quad\text{for every }t\in S\cup\{e\}\,.
        \end{equation} 
        \noindent\textit{Case $s_0\in B^D$.}  We  define 
        \begin{equation}
            \Phi(t)\defeq s_0 t\qquad\text{for $t\in S\cup\{e\}$}
        \end{equation}  and we notice that
        \begin{equation}
d_S(t,\Phi(t))=d_S(t,s_0t)=|t^{-1}s_0 t|_S= 1\qquad\text{for every }t\in S\cup\{e\}\,,
        \end{equation}
        as $B^D$ is conjugation invariant.\\
        \textit{Case $s_0\in U_b$}. We write $U_b=\{s_0,\bar s_0\}$ and we define
            \begin{alignat}{3}
                \Phi(e)&\defeq s_0\\
                \Phi(t)&\defeq ts_0\qquad&&\text{for $t\in B^D\setminus\{e\}$}\\
                \Phi(s_0)&\defeq s_0 \bar s_0^{-1}\qquad &&\\
                \Phi(\bar s_0)&\defeq e\\
                \Phi(t)&\defeq  s_0\theta_{c,b^{-1}c}(t) \qquad&&\text{for $t\in U_{c}$, $c\notin\{e,b\}$}\,.
            \end{alignat}
            Notice first that for $t\in B^D\setminus\{e\}$, $ts_0=s_0 s_0^{-1}ts_0\in s_0 (B^D\setminus\{e\})$, by conjugation invariance.            Now, $\Phi^{-1}(s_0)=\{e\}$ and $\Phi(S)=s_0S$ bijectively, so that the first condition of \eqref{vecdscc}   follows.     
            For the second condition of \eqref{vecdscc}, notice that $s_0,\bar s_0^{-1}\in S$ and  that, for $t\in U_c$,  $c\notin \{e,b\}$,
            \begin{equation}
                \pi(t^{-1} s_0\theta_{c,b^{-1}c}(t))=e\,,
            \end{equation}
            so that $t^{-1} s_0\theta_{c,b^{-1}c}(t)\in B^D$, for $D$ large enough (notice that in the above no element of $B^D$ appears).
\qed

\section{Transitive graphs}
\subsection{The case of finite vertex stabilizers}

First, we recall that if $G=(V,E)$ is a graph and $\Gamma\le \Aut(G)$, the vertex stabilizer of $\Gamma$ at $x\in V$ is defined as
\begin{equation}
    \Gamma_x\defeq \{\gamma\in\Gamma: \gamma(x)=x\}\,.
\end{equation}
Notice that if  $\Gamma$ acts transitively on $V$, then all the vertex stabilizers are conjugate by an element of $\Gamma$. We first prove our main result in this setting, Theorem \ref{notationsimplify}, under the additional assumption of finite vertex stabilizers. We show in the next subsection how this extra assumption can be removed.
\begin{prop}\label{notationsimplify:res}
    Let $G=(V,E),\Gamma, P$ be as in the statement of Theorem \ref{notationsimplify}. Assume, in addition, that $\Gamma$ has finite vertex stabilizers. Then  $G$ is quasi-isometric to $\ZZ^k$, for some $k\in\NN$.
\end{prop}
\begin{proof}
	We start by recalling  the construction of \cite[Theorem 4]{Sabidussi}. First, fix $o\in V$ and let $\Gamma_o$ be the vertex stabilizer at $o$, with $\ell\defeq |\Gamma_o|\in\mathbb{N}$. Then, we set $\ell V\defeq V\times \Gamma_o\simeq V\times\{1,\dots,\ell\}$.   We fix a section $\sigma:V\rightarrow\Gamma$ of the orbit map $\Gamma\ni\gamma\mapsto \gamma(o)\in V$ and we consider
    \begin{equation}\notag
        \ell V\ni (x,i)\mapsto \sigma(x)i\in\Gamma\,,
    \end{equation} 
    whose inverse is clearly
    \begin{equation}\notag
        \Gamma\ni \gamma\mapsto (\gamma (o),  \sigma(\gamma(o))^{-1}\gamma)\in \ell V\,.
    \end{equation}
    We then consider the graph $\ell G=(\ell V,\ell E)$, where $(x,i)\sim (y,j)$ (in $\ell G$) if and only if $x\sim y$ (in $G$). It is proved in \cite[Theorem 4]{Sabidussi} that $\ell G=(\ell V, \ell E)$ is the Cayley graph of $\Gamma$ with respect to a {finite}, symmetric, set of generators, in particular, that the graph isomorphism is the map  above. 
    	Of course, $\ell G$ is quasi-isometric to $G$ and hence has still polynomial growth, so that $\Gamma$ is virtually nilpotent by Gromov's Theorem \cite{MR623534}.
Take now $\gamma\in\Gamma$. Using the isomorphism above,
    \begin{equation}\label{isoinv}
        \gamma (x,i)=(\gamma (x), \sigma(\gamma (x))^{-1}\gamma\sigma(x)i)\,\qquad\text{for every $\gamma \in \Gamma$ and $(x,i)\in \ell V$}\,,
    \end{equation}
    where the action on the second coordinate is not important for us. This is to say the group multiplication on $\Gamma\cong \ell V$ corresponds to the action of $\Gamma\le\Aut(G)$ on $V$.

	We now build a Markov process on $\ell G$, whose transition matrix is denoted by $P_{\ell G}$.  We define
	\begin{equation}\notag
		 P_{\ell G}((x,i),(y,j))\defeq 	\frac{P(x,y)}{\ell}		\qquad\text{for every $(x,i),(y,j)\in\ell V$}\,,
	\end{equation}
	so that $\pi_{*}( P_{\ell G}((x,i),\,\cdot\,))=P(x,\,\cdot\,)$ for every $(x,i)\in \ell V$, where $\pi:\ell V\rightarrow V\times \{1\}\simeq V$ denotes  the natural projection. It follows that $\pi_{*}( P_{\ell G}^n((x,i),\,\cdot\,))=P^n(x,\,\cdot\,)$  for every $n\in\mathbb{N}$,
	so that 
	\begin{equation}\notag
		W_1^{\ell G}( P_{\ell G}^n((x,i),\,\cdot\,),P^n(x,\,\cdot\,))\le 2 \qquad\text{for every $(x,i)\in \ell V$ and $n\in\mathbb{N}$}\,,
	\end{equation}	
    where we used $V\simeq V\times\{1\}$.
	Also, it is clear that 
	\begin{equation}\notag
		W_1^{\ell G}(P^n(x,\,\cdot\,),P^n(y,\,\cdot\,))=W_1^{ G}(P^n(x,\,\cdot\,),P^n(y,\,\cdot\,))\le d(x,y)\qquad\text{for every $x,y\in V$ and $n\in\mathbb{N}$}\,,
	\end{equation}
	where we used that $G$ has non-negative Ollivier--Ricci curvature with the argument of Remark \ref{rem:vfdsc}. The two equations above imply that 
	\begin{equation}\label{vefdscvcsc}
		W_1^{\ell G}( P_{\ell G}^n((x,i),\,\cdot\,),P_{\ell G}^n((y,j),\,\cdot\,))\le d((x,i),(y,j))+4\qquad\text{for every $(x,i),(y,j)\in \ell V$ and $n\in\mathbb{N}$}\,.
	\end{equation}
	Moreover, $P_{\ell G}$ is diagonally invariant for $\Gamma$, in the sense that 
	\begin{equation}\notag
		P_{\ell G}( \gamma(x,i),\gamma(y,j))=P_{\ell G}( (x,i),(y,j))\qquad\text{for every $(x,i),(y,j)\in\ell G$ and $\gamma\in\Gamma$}\,,
	\end{equation}
	which follows from \eqref{isoinv}. Finally, it is immediate to see that $P_{\ell G}$ retains the properties of irreducibility, symmetry and laziness and bounded-range interaction satisfied by $P$.

	Hence, we can apply  our main result, Theorem \ref{vecdscs:res}, to conclude that $\Gamma$ is virtually abelian, so that $\ell G$ is quasi-isometric to $\mathbb{Z}^k$ (with some translation invariant metric), for some $k\in\mathbb{N}$. 
\end{proof}

\subsection{Proof of Theorem \ref{notationsimplify}}	
	We first recall that if we have $H\unlhd \Aut(G)$, then we can naturally define the graph $G/H$, whose vertex set is $V/H$ and such that $Hx\sim Hy$ (in $G/H$) if and only if there exist $x'\in Hx,y'\in Hy$ such that $x'\sim y'$ (in $G$).
	As $G$ has polynomial growth by assumption, we can apply Trofimov's Theorem  (\cite{MR735714}, see also \cite{WOESS1991373}), as presented in \cite[Theorem 2.1]{TesseraTointon}, to obtain a subgroup $H\unlhd \Aut(G)$ such that the projection $V\rightarrow V/H$ has finite fibres, $\Aut(G)$ induces a transitive action on $G/H$ corresponding to the virtually nilpotent group $\Aut(G)/H\le\Aut(G/H)$ and $\Aut(G)/H$  has  finite  vertex stabilizers.
	
	The projection of the Markov process on $G$ is a Markov process  on $G/H$, whose transition matrix is 
	\begin{equation}\notag
		P_{G/H}(Hx, Hy)\defeq \sum_{y'\in Hy} P(x',y')\,,
	\end{equation}
	where $x'$ is any element of $Hx$. It is part of the claim above that this is a good definition. Similarly, we have the diagonal invariance 
	\begin{equation}\notag
		P_{G/H}(gHx, gHy)=		P_{G/H}(Hx, Hy)\qquad\text{for every }g\in \Aut(G)\,.
	\end{equation}
 Notice that $P_{G/H}$ retains laziness (as $P_{G/H}(Hx,Hx)\ge 1/2$) and symmetry (i.e., $P_{G/H}(Hx,Hy)=P_{G/H}(Hy,Hx)$), that $P_{G/H}$ is compactly supported (as $P_{G/H}(Hx, Hy)=0$ for all but finitely many cosets $Hy$) and irreducible. 
 Notice finally that, for every $x\sim y\in V$,
	\begin{equation}\notag 
			W^{G/H}_1(P_{G/H}(Hx,\,\cdot\,), P_{G/H}(Hy,\,\cdot\,))\le W^G_1(P(x,\,\cdot\,), P(y,\,\cdot\,))\,
	\end{equation}
as we can take the push-forward through the projection of optimal plans. Using that $G$ has non-negative Ollivier--Ricci curvature, we see that $G/H$, endowed with the diagonally invariant Markov kernel $P_{G/H}$, has non-negative Ollivier--Ricci curvature as well.
Also, as the projection map $V\rightarrow V/H$ is a quasi-isometry, we see that $G$ is quasi-isometric to $G/H$. 

The conclusion follows from Proposition \ref{notationsimplify:res}.\qed

\end{document}